\documentclass[a4paper,11pt,abstracton]{scrartcl}

\usepackage[margin=25mm]{geometry}
\usepackage{natbib}
\usepackage{etex}
\usepackage{scrhack}
\usepackage{cmap}
\usepackage[ngerman,english]{babel}
\usepackage[utf8]{inputenc}
\usepackage{graphicx}
\usepackage{caption}
\usepackage{subfigure}
\usepackage{epstopdf}
\usepackage{mathtools}
\usepackage{amsmath}
\usepackage{amsthm}
\usepackage{amssymb}
\usepackage{comment}
\usepackage{multirow}
\usepackage{tikz}
\usetikzlibrary{arrows,intersections}
\usetikzlibrary{shapes,snakes}
\usepackage[english,vlined,ruled]{algorithm2e}
\usepackage{hyperref}
\usepackage{lipsum}
\usepackage{float}
\usepackage{tabularx}
\usepackage{booktabs}
\usepackage{longtable}
\usepackage{tabu}
\usepackage{environ}
\usepackage{mdframed}
\usepackage{datetime}
\usepackage{enumerate}
\usepackage{authblk}
\usepackage[capitalize, noabbrev]{cleveref}

\newtheorem{theorem}{Theorem}
\newtheorem{cor}{Corollary} 
\newtheorem{lem}{Lemma}
\newtheorem{obs}{Observation}

\theoremstyle{definition}

\theoremstyle{remark}

\mdfdefinestyle{probframe}{skipabove=0pt,skipbelow=0pt,innerleftmargin=0pt,
innerrightmargin=5pt,innertopmargin=0pt, innerbottommargin=5pt}

\makeatletter
\newcommand{\problemtitle}[1]{\gdef\@problemtitle{#1}}
\newcommand{\probleminput}[1]{\gdef\@probleminput{#1}}
\newcommand{\problemquestion}[1]{\gdef\@problemquestion{#1}}
\NewEnviron{problem}{
    \medskip
\begin{mdframed}[style=probframe]
  \problemtitle{}\probleminput{}\problemquestion{}
  \BODY
  \par\addvspace{.5\baselineskip}
  \noindent
  \begin{tabularx}{\textwidth}{@{\hspace{\parindent}} l X c}
    \multicolumn{2}{@{\hspace{\parindent}}l}{\@problemtitle} \\
    \textbf{Input:} & \@probleminput \\
    \textbf{Question:} & \@problemquestion
  \end{tabularx}
  \par\addvspace{.5\baselineskip}
\end{mdframed}
    \medskip
}
\makeatother

\newcommand{\X}{\mathcal{X}}
\newcommand{\cU}{\mathcal{U}}

\title{Recoverable Robust Representatives Selection Problems with Discrete Budgeted Uncertainty}

    \author[1]{Marc Goerigk\footnote{marc.goerigk@uni-siegen.de}}
    \author[2,3]{Stefan Lendl\footnote{lendl@math.tugraz.at; corresponding author}}
    \author[2]{Lasse Wulf\footnote{wulf@math.tugraz.at}}
	
	\affil[1]{Network and Data Science Management, University of Siegen, Unteres Schloß 3, 57072 Siegen, Germany}	
	\affil[2]{Institute of Discrete Mathematics, Graz University of Technology, Steyrergasse 30/II, 8010 Graz, Austria}
  \affil[3]{Institute of Operations und Information Systems, University of Graz, Universitätsstraße 15, 8010 Graz, Austria}
	\date{}

\begin{document}
\maketitle

\begin{abstract} 
Recoverable robust optimization is a multi-stage approach, in which it is possible to adjust a first-stage solution after the uncertain cost scenario is revealed. We analyze this approach for a class of selection problems. The aim is to choose a fixed number of items from several disjoint sets, such that the worst-case costs after taking a recovery action are as small as possible. The uncertainty is modeled as a discrete budgeted set, where the adversary can increase the costs of a fixed number of items.

While special cases of this problem have been studied before, its complexity has remained open. In this work we make several contributions towards closing this gap. We show that the problem is NP-hard and identify a special case that remains solvable in polynomial time. We provide a compact mixed-integer programming formulation and two additional extended formulations. Finally, computational results are provided that compare the efficiency of different exact solution approaches.
\end{abstract}


\noindent\textbf{Funding:} This work was supported by the Austrian Science Fund (FWF): W1230.

\section{Introduction}\label{sec:intro}

Most optimization problems in practice are uncertain. To handle such problems under uncertainty, a vibrant field of research has been developed, including such approaches as fuzzy optimization (see e.g. \cite{lodwick2010fuzzy}), stochastic programming (see e.g. \cite{birge2011introduction}), or robust optimization (see e.g. \cite{kasperski2016robust}). Treating uncertainty in an optimization problem typically increases its computational complexity, which means that a problem that is simple under known problem data may become challenging when the data is not known precisely.

In this paper we consider one such problem, which is the representatives multi-selection problem (RMSP). We are given several disjoint sets of items, and need to choose a specified number of items from each set. The aim is to minimize a linear cost function in the items. More formally, the problem can be described as
\[ \min_{\pmb{x}\in\X} \pmb{c}^t \pmb{x} \]
with
\[ \X = \{ \pmb{x}\in\{0,1\}^n : \sum_{i\in T_j} x_i = p_j\ \forall j\in[K]\} \]
and $T_1 \cup \ldots \cup T_K = [n]$ forming a partition of the item set into disjoint sets called parts, where we use the notation $[n]=\{1,2,\ldots,n\}$. We will also write $n_j = |T_j|$ to denote the size of part $j\in[K]$. The special case of $K=1$ is known as the selection problem (see \cite{kasperski2015robust}), while the case with $p_j=1$ has been studied as the representatives selection problem (see \cite{kasperski2015approximability}) or weighted disjoint hitting set problem (see \cite{busing2011phd}). Note that it is easy to solve this problem; sorting items of each part by their costs is already sufficient.

To follow a robust optimization approach for this problem, we need to specify an uncertainty set $\cU$ containing all cost scenarios against which we wish to prepare. The classic (single-stage, min-max) approach is then to consider the problem
\[ \min_{\pmb{x}\in\X} \max_{\pmb{c}\in\cU} \pmb{c}^t\pmb{x} \]
A drawback of this approach is that it does not incorporate the possibility to react once scenario information becomes available. To alleviate this, two-stage approaches have been introduced, in particular adjustable robust optimization (see \cite{yanikouglu2019survey}) and recoverable robust optimization (see \cite{liebchen2009concept}). In the latter approach, we fix a complete solution in the first stage, and can slightly adjust it after the scenario has been revealed.

Different types of uncertainty sets have been proposed in the literature, including discrete uncertainty, interval uncertainy, or ellipsoidal uncertainty (see \cite{goerigk2016algorithm}). Particularly successful has been the so-called budgeted uncertainty as first introduced in~\cite{bertsimas2003robust}, where only a bounded number of cost coefficients can deviate from their nominal values. That is, the set of possible cost scenarios is given as
\[ \cU = \{ \pmb{c}\in\mathbb{R}^n : c_i = \underline{c}_i + d_i \delta_i, \delta_i\in\{0,1\}, \sum_{i\in[n]} \delta_i \le \Gamma \} \]
for some integer $\Gamma$. We define $\overline{c}_i = \underline{c}_i + d_i$.
We refer to this type of uncertainty as discrete budgeted uncertainty. It is also possible to define continuous budgeted uncertainty, where we allow the deviations $\delta_i$ to be continuous within the interval $[0,1]$. Discrete budgeted uncertainty hence contains only the extreme points of the continuous budgeted uncertainty polyhedron. In the case of single-stage robust optimization, the continuous and discrete variants are equivalent and the robust counterpart can be solved in polynomial time if this is the case for the underlying nominal problem; both observations do not necessarily hold for two-stage robust optimization.

We use a recoverable robust approach, where we buy a solution in the first stage with known costs $\pmb{C}$, and then can adjust this solution after the uncertain second-stage costs have been revealed. 
We denote the maximum number of items that can be changed by $k$.
Let $\Delta(\pmb{x},\pmb{y}) = \sum_{i\in[n]} |x_i-y_i|$ denote the Hamming distance, and let $R(\pmb{x}) = \{ \pmb{y}\in\X : \Delta(\pmb{x},\pmb{y})\le 2k\}$ be the set of recovery solutions for some integer $k$.
We define the following problems. In the incremental problem, we are given $\pmb{x}$ and $\pmb{c}$, and we solve
\[ \textsc{Inc}(\pmb{x},\pmb{c}) = \min_{\pmb{y}\in R(\pmb{x})} \sum_{i\in[n]} c_i y_i. \]
This represents the recovery step after the solution $\pmb{x}$ has been fixed, and the scenario has been revealed. A layer above this is the adversarial problem, where given $\pmb{x}$, we solve
\[ \textsc{Adv}(\pmb{x}) = \max_{\pmb{c}\in\cU}\ \textsc{Inc}(\pmb{x},\pmb{c}). \]
Finally, the recoverable robust representatives multi-selection problem (RRRMSP) is to solve
\[ \textsc{Rec} = \min_{\pmb{x}\in\X} \sum_{i\in[n]} C_ix_i + \textsc{Adv}(\pmb{x}). \]
We sometimes refer to this as the recoverable robust problem for short. The special case of the recoverable robust representatives selection problem (RRRSP), where $p_j=1$,  was first considered in the PhD thesis of \cite{busing2011phd}. The complexity of the case of discrete budgeted uncertainty was highlighted as being open.

The special case $K=1$, i.e., the recoverable robust selection problem with discrete budgeted uncertainty, was previously considered in~\cite{chassein2018recoverable}. It was shown that the adversarial problem can be solved in polynomial time, and a compact formulation for the recoverable robust problem was derived. The complexity of the problem remained open. So far, neither positive nor negative complexity results for the RRRMSP or RRRSP have been derived, despite the problems being open for nearly 10 years.

Other variants of robust selection problems have been considered in the literature as well. In \cite{averbakh2001complexity}, a polynomial time algorithm was presented for the selection problem with a min-max regret objective and interval uncertainty. The algorithm complexity was further improved in \cite{conde2004improved}. In \cite{dolgui2012min}, min-max and min-max regret representatives selection problem with interval and discrete uncertainty were considered and these results were further refined in \cite{deineko2013complexity}. More recently, \cite{goerigk2019robust} consider two-stage representatives selection problems with general convex uncertainty sets. Furthermore, the setting of recoverable robustness has been applied to other combinatorial problems as well, including knapsack (see \cite{busing2011recoverable}), shortest path (see \cite{busing2012paths}) and spanning tree (see \cite{hradovich2017recoverable-MST}) problems.

A natural assumption one may make about the RRRSP is that dominated items should not be packed in an optimal solution. That is, if there is $i\in T_j$ such that $C_i > C_{i'}$, $\underline{c}_i > \underline{c}_{i'}$ and $\overline{c}_i > \overline{c}_{i'}$ for some other item $i'\in T_j$, then we can assume $x_i=0$ in an optimal solution. We give an example demonstrating that this is in fact not the case, underlining that the recoverable robust problem, despite having a seemingly easy structure, is more complex than it appears.
In Table~\ref{tab:example} we show the data of an example RRRSP with $K=2$, $n_1 = n_2 = 2$ and $p_1 = p_2 = 1$. We further have $\Gamma=k=1$. 

\begin{table}[htb]
\begin{center}
\begin{tabular}{r|rr|rr}
 & \multicolumn{2}{c|}{$T_1$} & \multicolumn{2}{c}{$T_2$} \\
 & 1 & 2 & 3 & 4 \\
 \hline
$C_i$ & 1 & 5 & 8 & 7 \\
$\underline{c}_i$ & 10 & 7 & 9 & 4 \\
$\overline{c}_i$ & 19 & 17 & 19 & 13
\end{tabular}
\end{center}
\caption{Example problem with $\Gamma=k=1$.\label{tab:example}}
\end{table}

A natural candidate solution is to pick items 1 and 4. This choice has first-stage costs of 8. A worst-case scenario is that the costs of item 4 are increased, which forces us to respond by exchanging item 4 for item 3. The second-stage costs are thus $10+9=19$, with total costs $8 + 19 = 27$. Choosing items 2 and 4 results in an objective value of $12+16=28$.

Now consider the solution where we pack items 1 and 3. A worst-case attack is now on item 1, which results in an optimal recovery of exchanging items 1 and 2. The total costs of this solution are $9+16 = 25$. In fact, this is the unique optimal solution to the problem.
Note that item 3 is dominated by item 4, being worse in every cost coefficient.

\emph{Our contribution}. In Section~\ref{sec:adv} we show that it is possible to solve the adversarial problem in polynomial time. In Section~\ref{sec:hardness} we solve a long-standing open problem by showing that the recoverable robust representatives selection problem from \cite{busing2011phd} is NP-hard, even for $p_j=1$ and $n_j=2$. We further show that the recoverable and two-stage robust selection problem variants from \cite{chassein2018recoverable}, where $K=1$, are NP-hard as well. In Section~\ref{sec:special} we show that a special case of the problem can be solved in polynomial time. Here we assume that $n_j=2$, $\Gamma=1$ and $k=1$. That is, each part contains exactly two items, of which one must be chosen. The adversary can increase costs once, and we can recover by exchanging a single item. The idea to prove that this case can be solved in polynomial time is based on the following observation. Consider any min-max problem with $S$ scenarios:
\[ \min_{\pmb{x}\in\X} \max_{s\in[S]} \sum_{i\in[n]} c^s_i x_i\]
This problem can be equivalently written as
\[ \min_{s\in[S]} \min_{\pmb{x}\in\X^s} \sum_{i\in[n]} c^s_i x_i \]
where
\[ \X^s = \{ \pmb{x}\in\X : \sum_{i\in[n]} c^s_i x_i \ge \sum_{i\in[n]} c^{j}_i x_i\ \forall j\in[S] \} \]
that is, we guess the worst-case scenario $s$, but restrict the set of feasible solutions to those where $s$ is indeed the worst case. As far as we are aware, such an approach has not been successfully applied before. 
  We consider problem models in Section~\ref{sec:mips}, where we use insight on the adversarial problem from Section~\ref{sec:adv} to derive a compact mixed integer programming formulation, i.e., a formulation as a mixed integer program of polynomial size. As the number of constraints and variables is of the order $O(n^5)$, we also discuss different iterative solution approaches. 
We present computational experiments in Section~\ref{sec:experiments}, comparing different exact solution approaches developed in this paper. Finally, Section~\ref{sec:conclusions} concludes the paper, and further research questions are pointed out.

\section{Adversarial problem}
\label{sec:adv}

In this section, we show that the adversarial problem $\textsc{Adv}(\pmb{x})$ can be solved in polynomial time. To this end, we first derive a compact mixed-integer programming formulation of the problem, and use a decomposition argument to find a polynomial number of subproblems, which are then combined using a dynamic program.

To derive a model for 
\[ \textsc{Adv}(\pmb{x}) = \max_{\pmb{c}\in\cU}\ \textsc{Inc}(\pmb{x},\pmb{c}) \]
we first consider the incremental problem and model it as the following integer program.
\begin{subequations}
\label{eq:i}  
\begin{align}
\textsc{Inc}(\pmb{x},\pmb{c}) = \min\ & \sum_{i\in[n]} c_i y_i \label{i-0} \\
\text{s.t. } & \sum_{i\in T_j} y_i = p_j  & \forall j\in[K] \label{i-1}\\
& \sum_{i\in[n]} x_i y_i \ge P-k \label{i-2}\\
& y_i\in\{0,1\} & \forall i\in[n] \label{i-3}
\end{align}
\end{subequations}
where $P = \sum_{j \in [K]} p_j$ is the total number of elements to select. Variable $y_i$ denotes if item $i$ is contained in the recovery solution. Constraint~\eqref{i-2} ensures that we must use at least $P-k$ items from the first-stage solution $\pmb{x}$, i.e., at most $k$ items can be exchanged for other items. As $\pmb{x}$ is fixed in this context, model~(\ref{eq:i}) is an integer linear program. Note that the coefficient matrix is totally unimodular (which can be easily seen by applying, e.g., the Ghouila-Houri criterion (see \cite{GH62})). We dualize its linear programming relaxation by introducing variables $\alpha_j$ for constraints~\eqref{i-1}, variable $\beta$ for constraint~\eqref{i-2}, and variables $\gamma_i$ for constraints~\eqref{i-3} (which become $y_i \le 1$ in the relaxation). Using the dual of $\textsc{Inc}$ and by exploiting weak and strong duality, we can construct the following compact formulation for the adversarial problem.
\begin{subequations}
\label{eq:ad}
\begin{align}
\textsc{Adv}(\pmb{x}) = \max\ & \sum_{j\in[K]} p_j \alpha_j + (P-k)\beta - \sum_{i\in[n]} \gamma_i \\
\text{s.t. } & \sum_{i\in[n]} \delta_i \le \Gamma \label{ad-1} \\
& \alpha_j + x_i \beta \le \underline{c}_i + d_i\delta_i + \gamma_i & \forall j\in[K], i\in T_j \label{ad-2}\\
& \delta_i\in\{0,1\} & \forall i\in[n] \label{ad-3}\\
& \beta \ge 0 \label{ad-4}\\
& \gamma_i \ge 0 & \forall i\in[n] \label{ad-5}
\end{align}
\end{subequations}
We use $\delta_i$ as a binary variable indicating which item costs should be increased. Constraint~\eqref{ad-1} ensures that the total number of items with increased costs should be less or equal to $\Gamma$. Constraints~\eqref{ad-2} are the dual constraints to variables $y_i$ in problem $\textsc{Inc}$.

In the following we show that $\textsc{Adv}$ can be solved in polynomial time. To this end, we use an enumeration argument to decompose $\textsc{Adv}$ into simpler subproblems.

Let us first assume that we fix variable $\beta$ to some value, and decide how many items $\Gamma_j$ per part $T_j$ should have increased costs with $\sum_{j\in[K]} \Gamma_j \le \Gamma$. Then we have
\[ \textsc{Adv}(\pmb{x}) = \max_{\beta} \max_{\Gamma_1,\ldots,\Gamma_K} \sum_{j\in[K]} \textsc{Adv}_j(\pmb{x},\beta,\Gamma_j) \]
with subproblems
\begin{align*}
\textsc{Adv}_j(\pmb{x},\beta,\Gamma_j) = (P-k)\beta + \max\ &  p_j \alpha_j - \sum_{i\in T_j} \gamma_i \\
\text{s.t. } & \sum_{i\in[n]} \delta_i \le \Gamma_j \\
& \alpha_j + x_i \beta \le \underline{c}_i + d_i\delta_i + \gamma_i & \forall i\in T_j \\
& \delta_i\in\{0,1\} & \forall i\in T_j \\
& \gamma_i \ge 0 & \forall i\in T_j
\end{align*}
In particular, for fixed $\beta$ and $\Gamma_j$, we can decompose the problem and consider each part $T_j$ separately. Note that in an optimal solution to $\textsc{Adv}_j(\pmb{x},\beta,\Gamma_j)$, we can assume that
\[ \gamma_i = [\alpha_j + x_i \beta - \underline{c}_i - d_i\delta_i ]_+ \]
where we use the notation $[x]_+ = \max\{x,0\}$ for the positive part of a value. Using this observation, we can rewrite the problem as 
\begin{subequations}
\label{eq:a}  
\begin{align}
\textsc{Adv}_j(\pmb{x},\beta,\Gamma_j) = (P-k)\beta + \max\ &  p_j \alpha_j - \sum_{i\in T_j} [\alpha_j + x_i \beta - \underline{c}_i - d_i\delta_i]_+ \label{a-0} \\
\text{s.t. } & \sum_{i\in[n]} \delta_i \le \Gamma_j \label{a-1}\\
& \delta_i\in\{0,1\} & \forall i\in T_j \label{a-2}
\end{align}
\end{subequations}
For any fixed choice of $\pmb{\delta}$, the remaining problem is piecewise linear in variable $\alpha_j$. We can conclude that an optimal value for $\alpha_j$ is at one of the kink points, where the slope of the piecewise linear function changes. Hence, there is an optimal $\alpha_j \in A_j(\beta)$ with
\begin{align*}
A_j(\beta) &= A^1_j \cup A^2_j \cup A^3_j(\beta) \cup A^4_j(\beta) \text{, where}  \\
A^1_j &= \{ \underline{c}_i : i\in T_j \} \\
A^2_j &= \{  \underline{c}_i + d_i  : i \in T_j \} \\
A^3_j(\beta) &= \{\underline{c}_i - \beta : i\in T_j\} \\
A^4_j(\beta) &= \{\underline{c}_i +d_i - \beta : i \in T_j \} 
\end{align*}
For a fixed choice of $\alpha_j$, problem~(\ref{eq:a}) is equivalent to a selection problem, which can be solved in $O(|T_j|)$ time. Furthermore, we have $|A_j(\beta)|=O(|T_j|)$. Overall, we can calculate
\[ \textsc{Adv}_j(\pmb{x},\beta,\Gamma_j) = \max_{\alpha_j \in A_j(\beta)} \textsc{Adv}_j(\pmb{x},\beta,\Gamma_j,\alpha_j) \]
in $O(|T_j|^2)$, where $\textsc{Adv}_j(\pmb{x},\beta,\Gamma_j,\alpha_j)$ denotes problem~(\ref{eq:a}) with fixed choice of $\alpha_j$.

Now let us assume that for each part $j$, we guess the value of $\alpha_j$. We partition $[K]$ into the sets $V$ and $W$ accordingly, where $\alpha_j= v_j$ for $j\in V$ (in case $\alpha_j\in A^1_j\cup A^2_j$) and
$\alpha_j= v_j - \beta$ for $j\in W$ (in case $\alpha_j \in A^3_j(\beta)\cup A^4_j(\beta)$).
We find that
\begin{align*}
\textsc{Adv}(\pmb{x}) = \max\ & (P-k)\beta + \sum_{j\in V} p_j v_j + \sum_{j\in W} p_j (v_j - \beta) - \sum_{i\in[n]} \gamma_i \\
\text{s.t. } & \sum_{i\in[n]} \delta_i \le \Gamma \\
& v_j + x_i \beta \le \underline{c}_i + d_i\delta_i + \gamma_i & \forall j\in V, i\in T_j \\
& v_j - \beta + x_i \beta \le \underline{c}_i + d_i\delta_i + \gamma_i & \forall j\in W, i\in T_j \\
& \delta_i\in\{0,1\} & \forall i\in[n] \\
& \beta \ge 0 \\
& \gamma_i \ge 0 & \forall i\in[n]
\end{align*}
As before, we can assume that in an optimal solution we have
\begin{align*}
\gamma_i &= [v_j + x_i \beta - \underline{c}_i - d_i\delta_i]_+ & \forall j\in V, i \in T_j \\
\gamma_i &= [v_j -\beta + x_i \beta - \underline{c}_i - d_i\delta_i]_+ & \forall j\in W, i \in T_j 
\end{align*}
Using this property, $\textsc{Adv}(\pmb{x})$ can be rewritten as:
\begin{align*}
\textsc{Adv}(\pmb{x}) = \max\ & (P-k)\beta + \sum_{j\in V} p_j v_j + \sum_{j\in W} p_j (v_j - \beta) \\
& - \sum_{j\in V} \sum_{i \in T_j} [v_j + x_i \beta - \underline{c}_i - d_i\delta_i]_+ - \sum_{j\in W} \sum_{i \in T_j} [v_j -\beta + x_i \beta - \underline{c}_i - d_i\delta_i]_+ \\
\text{s.t. } & \sum_{i\in[n]} \delta_i \le \Gamma \\
& \delta_i\in\{0,1\} & \forall i\in[n] \\
& \beta \ge 0 
\end{align*}
Similar to the reasoning for $\alpha_j$, we find that $\textsc{Adv}(\pmb{x})$ is piecewise linear in $\beta$, and conclude that there is an optimal $\beta$ which is equal to one of the kink points
\[ B(\alpha_1,\ldots,\alpha_K) = \{ \underline{c}_i + d_i\delta_i - v_j : j\in[K], i\in T_j\} \cup\{v_j - \underline{c}_i - d_i\delta_i : j\in[K], i\in T_j \} \cup \{0\}.\]
Taking into account the possible values for $v_j$, we hence conclude that it is sufficient to consider $\beta$ from the set
\[ B = \{ \underline{c}_i + d_i b_i - \underline{c}_j - d_j b_j : i,j\in[n], b_i,b_j\in\{0,1\} \} \cup \{0\}. \]
Note that $|B|\in O(n^2)$. 

We can now show that $\textsc{Adv}(\pmb{x})$ can be solved in polynomial time. Note that directly enumerating all combinations of $\alpha_j$ and $\Gamma_j$ would require exponential time, which is why we combine our structural observations with a dynamic program to show the following result.

\begin{theorem}\label{th:adv}
The adversarial problem of RRRMSP with discrete budgeted uncertainty can be solved in strongly polynomial time.
\end{theorem}
\begin{proof}
We first enumerate all values $\beta\in B$. For each $j\in K$ and each $\Gamma_j\in\{0,\ldots,\Gamma\}$,  we calculate $\textsc{Adv}_j(\pmb{x},\beta,\Gamma_j)$ by enumerating over all possible values of $\alpha_j$. 

We then use the following dynamic program. Let 
\begin{align*}
F_{\pmb{x},\beta}(K',\Gamma') := \max \ & \sum_{j\in[K']} \textsc{Adv}_j(\pmb{x},\beta,\Gamma_j) \\
\text{s.t. } & \sum_{j\in[K']} \Gamma_j = \Gamma' \\
& \Gamma_j \in \mathbb{N}_0
\end{align*}
denote the maximum adversary value that is achievable using only the first $K'$ parts and a budget of $\Gamma'$. We have $F_{\pmb{x},\beta}(1,\Gamma') = \textsc{Adv}_1(\pmb{x},\beta,\Gamma')$. Using the recursion
\[ F_{\pmb{x},\beta}(K',\Gamma') = \max\{ F_{\pmb{x},\beta}(K'-1,\Gamma'-\Gamma_j) + \textsc{Adv}_{K'}(\pmb{x},\beta,\Gamma_j) : \Gamma_j=0,\ldots,\Gamma'\} \]
it is then possible to calculate all values of $F_{\pmb{x},\beta}$. As
\[ \textsc{Adv}(\pmb{x}) = \max_{\beta\in B} F_{\pmb{x},\beta}(K,\Gamma) \]
it is possible to solve the adversarial problem in polynomial time. More precisely, calculating all values $\textsc{Adv}_j(\pmb{x},\beta,\Gamma_j)$ for fixed $\beta$ requires $O( \sum_{j\in[K]} \Gamma|T_j|^2)$ time. The subsequent dynamic program needs to calculate $O(\Gamma K)$ many values, each of which requires $O(\Gamma)$ table lookups. So overall the runtime of this method is
\[ O(|B|(\sum_{j\in[K]} \Gamma|T_j|^2 + \Gamma^2 K)) = O(n^4\Gamma + n^2 \Gamma^2 K ) = O(n^5)\]
\end{proof}

\section{Hardness of representatives selection}
\label{sec:hardness}

While the adversarial problem of RRRMSP can be solved in polynomial time, the recoverable robust problem is hard already in the case of RRRSP, as the following result shows.

\begin{theorem}\label{th:hardness}
The decision version of the recoverable robust representatives selection problem (i.e., the case $p_j = 1$ for all $j$) with discrete budgeted uncertainty is NP-complete, even if we have $n_j = 2$ for all $j$.
\end{theorem}
 
\begin{proof}
The membership to NP follows from the fact that the adversary problem is in P (\cref{th:adv}), so it remains to show NP-hardness.

To show this, we reduce from the well-known NP-complete problem \textsc{Partition}. An instance of \textsc{Partition} consists of a multi-set $A = \{a_1, \ldots, a_{n}\}$. It is a yes-instance if there exists $A_0 \subseteq A$ such that $\sum A_0 = 1/2 \sum A$, where we denote by $\sum A$ the sum of elements in $A$. Given an instance $A$ of \textsc{Partition}, define $Q := 1/2 \sum A$ and let $M > \max_i 2a_i$ be some big integer. 
Consider the following instance $I$ of RRRSP, which is depicted in \cref{tab:hardness-reduction}.

\begin{table}[htb]
\begin{center}
\begin{tabular}{c|cc|c|cc|cc|c|cc}
 \multicolumn{1}{c}{} & \multicolumn{5}{c}{$\alpha$}& \multicolumn{5}{c}{$\beta$}\\[-1ex]
 \multicolumn{1}{c}{} & \multicolumn{5}{c}{\downbracefill} & \multicolumn{5}{c}{\downbracefill}\\[2ex]
 & \multicolumn{2}{c|}{$T_1$} &  & \multicolumn{2}{c|}{$T_{n}$} & \multicolumn{2}{c|}{$T_{n + 1}$} &  & \multicolumn{2}{c}{$T_{2n + 1}$} \\
 & 1 & 2 & $\cdots$ & 1 & 2 & 1 & 2 & $\cdots$ & 1 & 2\\
 \hline
$C_i$ & $a_1$ & 0 & $\cdots$ & $a_{n}$& 0 & $\infty$ & 0 & $\cdots$ & $\infty$ & 0 \\
$\underline{c}_i$ & 0 & 0 & $\cdots$ & 0 & 0 & 0 & $M$ & $\cdots$ & 0 & $M$ \\
$\overline{c}_i$ & 0 & $2a_1$ & $\cdots$ & 0 & $2a_{n}$ & 0 & $M + 2Q$ & $\cdots$ & 0 & $M+2Q$
\end{tabular}
\end{center}
\caption{Instance used in the hardness reduction for RRRSP with $n_j = 2$ and $p_j = 1$.\label{tab:hardness-reduction}}
\end{table}

The value $\infty$ is used to denote a constant that prevents any optimal solution from packing such an item; as the following analysis shows, $\infty > M+4Q$ is sufficient.
For the sake of readability, items are not numbered consecutively. Instead, we write $(j,i)$ to refer to item number $i$ in part $T_j$.
There are $2n + 1$ parts $T_1, \ldots, T_{2n+1}$, where $T_j = \{(j,1), (j,2)\}$ contains exactly the two items $(j,1)$ and $(j,2)$.  
Note that for $j \in \{1, \ldots, n\}$, the costs depend on $j$, and for $j \in \{n+1, \dots, 2n+1\}$, the costs are identical.

Let $\alpha := \{T_1, \ldots, T_{n}\}$ be the first $n$ parts and $\beta := \{T_{n+1}, \ldots, T_{2n+1}\}$ be the remaining parts. Finally, let $\Gamma = n+1$ and $k = n$. This completes our description of the instance $I$. We now claim that $\textsc{Rec}(I) \leq M + 3Q$ if and only if $A$ is a yes-instance of \textsc{Partition}.

To see the `if' part, assume there is $A_0 \subseteq A$ with $\sum A_0 = Q$. Then consider the binary vector $\pmb x$ resulting from choosing item $(j,2)$ in the parts $T_j \in \alpha$ and choosing item $(j,1)$ if $a_j \in A_0$, and otherwise item $(j, 2)$ in the parts $T_j \in \beta$. Now consider the adversarial stage for this vector, i.e., $\textsc{Adv}(\pmb x)$: Independent of which $n+1$ items the adversarial player attacks, in the recovery stage all $n$ recoveries will take place in $\beta$, due to the choice of $M$. Furthermore, items from $\beta$ which were attacked in the adversarial stage are prioritized in the recovery. Hence, if the adversary does not attack all $n+1$ items $(j,2)$ for $T_j \in \beta$, then any attack in $\beta$ will prove useless in the end. Therefore, the adversary only has two valid strategies:  (i) Attack all items $(j,2)$ for $T_j \in \beta$. This leads to a result of $M + 2Q + Q$ after the recovery stage. (ii) Attack all items $(j,2)$ for $T_j \in \alpha$ and waste the remaining attack. This leads to a result of $M + Q + 2Q$ after the recovery stage. We conclude $\textsc{Rec}(I) \leq M + 3Q$.

To see the `only if' part, assume that for all $A_0 \subseteq A$ we have $\sum A_0 \neq Q$. Let $\pmb x$ be some binary vector picked by the first-stage player. We show that $\textsc{Adv}(\pmb x) > M + 3Q$. If $x_{j1} = 1$ for some $T_j \in \beta$, we are immediately done, so assume $x_{j2} = 1$ for all $T_j \in \beta$. Let $A_1 := \{a_j \mid j \in [n], x_{j1} = 1\}$. There are two cases: If $\sum A_1 > Q$, the adversary can apply strategy (i), which  leads to an end result of strictly more than $M + 3Q$. If $\sum A_1 < Q$, the adversary can apply strategy (ii). After this, the selected items in $\alpha$ have total cost greater than $3Q$. Therefore, the end result is strictly more than $M + 3Q$.

\end{proof}

Note that we can also interpret the representatives selection problem with $p_j=1$ as a graph-based problem. Let $G=(V,E)$ with $V=\{0,1,\ldots,K\}$ and there are $n_j$ many parallel edges from node $j-1$ to $j$ for each $j\in[K]$, representing the choice of item from set $T_j$. In this graph, the RRRSP can be equivalently seen as a recoverable robust shortest path or minimum spanning tree problem. Theorem~\ref{th:hardness} now implies the following.

\begin{cor}
The recoverable robust shortest path and minimum spanning tree problems with discrete budgeted uncertainty are NP-hard, even on series-parallel graphs.
\end{cor}

To the best of our knowledge, previous complexity results on these problems required general graphs (see \cite{busing2012paths,nasrabadi2013robust}).

Recall that the selection problem is a special case of the representatives multi-selection problem with $K=1$, i.e., there is only one part from which items need to be chosen. The complexity of the recoverable robust selection problem with discrete budgeted uncertainty has remained open to this date \cite{chassein2018recoverable}. We show that this problem is NP-hard.

\begin{theorem}
The decision version of the recoverable robust selection problem (i.e., the case $K=1$) with discrete budgeted uncertainty is NP-complete, even if $k=1$.
\end{theorem}
\begin{proof}
As the problem can be formulated as a mixed-integer program (see \cite{chassein2018recoverable}), we know that it is contained in the class NP. To show NP-hardness, we use a reduction from \textsc{Partition} again. Let $A=\{a_1,\ldots,a_n\}$ be given values. The task is to identify a set $A_0\subseteq A$ such that $\sum A_0 = 1/2 \sum A =: Q$, where we use the notation as in the proof of Theorem~\ref{th:hardness}. We further require that $|A_0| = n/2$ (note that the problem remains NP-hard in this setting, see \cite{garey1979computers}) and we assume that $n\ge 6$.

Let $M$ be a big value. We construct the an instance of the recoverable robust selection problem as depicted in Table~\ref{tab:red2}, where we use $n'=2n+2$, $p=n/2+1$, $\Gamma=n$, and $k=1$.
There are four types of items. Items of type $\alpha$ have $C_i=\infty$, $\underline{c}_i=-(\frac{n}{2}-2)Q$, and $\overline{c}_i=\infty$ (note that negative costs can be removed from the problem by adding a suitable constant, as every feasible solution must choose the same number of items). There are $n$ such items. There is one item of type $\beta$ with $C_i=-M$, $\underline{c}_i=\overline{c}_i=\infty$. There is also one item of type $\gamma$ with $C_i=\infty$ and $\underline{c}_i=\overline{c}_i=0$. Finally, there are $n$ items of type $\epsilon$, one for each given value $a_i$, with $C_i = a_i$, $\underline{c}_i=0$, and $\overline{c}_i=Q-2a_i$.

\begin{table}[htb]
\begin{center}
\begin{tabular}{c|ccc|c|c|ccc}
\multicolumn{1}{c}{} & \multicolumn{3}{c}{$\alpha$} & \multicolumn{1}{c}{$\beta$} & \multicolumn{1}{c}{$\gamma$} & \multicolumn{3}{c}{$\epsilon$} \\[-1ex]
\multicolumn{1}{c}{} & \multicolumn{3}{c}{\downbracefill} & \multicolumn{1}{c}{\downbracefill} & \multicolumn{1}{c}{\downbracefill} & \multicolumn{3}{c}{\downbracefill}\\[2ex]
$i$ & 1 & $\dots$ & $n$ & $n+1$ & $n+2$ & $n+3$ & $\dots$ & $2n+2$ \\
\hline
$C_i$ & $\infty$ & $\dots$ & $\infty$ & $-M$ & $\infty$ & $a_1$ & $\dots$ & $a_n$ \\
$\underline{c}_i$ & $-(\frac{n}{2}-2)Q$ & $\dots$ & $-(\frac{n}{2}-2)Q$ & $\infty$ & 0 & 0 & $\dots$ & 0 \\
$\overline{c}_i$ & $\infty$ & $\dots$ & $\infty$ & $\infty$ & 0 & $Q-2a_1$ & $\dots$ & $Q-2a_n$ 
\end{tabular}
\caption{Instance used in the hardness reduction for recoverable robust selection.}\label{tab:red2}
\end{center}
\end{table}

The constant $\infty$ denotes a value that is sufficiently large to prevent any solution from packing an item with such costs (it suffices to set $\infty > nQ$).
We can assume that none of the items of type $\alpha$ or $\gamma$ are used in the first-stage solution. Furthermore, we assume that $M$ is sufficiently large, such that the item of type $\beta$ must be part of an optimal solution (it suffices to set $M > Q$). The remaining $n/2$ first-stage items are then chosen amongst items of type $\epsilon$. For the adversary, there are only two strategies to allocate the uncertainty budget $\Gamma$: In the first case, all items in $\alpha$ are made expensive. Then the first-stage solution will swap item $\beta$ for item $\gamma$. In the second case, all items in $\epsilon$ are made expensive. Then, item $\beta$ will be swapped for one of the items in $\alpha$. Note that any mixed strategy can be disregarded, as they will still result in item $\beta$ being swapped for one of the items of type $\alpha$ that remain cheap.

Let $X$ be the first-stage costs of a choice of items amongst $\epsilon$. We find that the total costs of such a solution are
\[ -M + X + \max\{ 0, \frac{n}{2}Q-2X-(\frac{n}{2}-2)Q\} = -M + \max\{ X, 2Q-X\} \]
Hence, there exists a solution to the recoverable robust selection problem with costs less or equal to $Q-M$ if and only if $X=Q$, which means that the \textsc{Partition} instance is a yes-instance.

\end{proof}

As a final hardness result, we remark that also the two-stage version of the selection problem with discrete budgeted uncertainty as considered in \cite{chassein2018recoverable} is NP-hard as well. The corresponding proof is given in the appendix.

\begin{theorem}\label{th:hardness3}
The decision version of the two-stage robust selection problem with discrete budgeted uncertainty is NP-complete, even if $k=1$.
\end{theorem}

\section{Polynomially solvable cases}

\label{sec:special}

We now consider the special case of representatives selection with $n_j=2$ and $p_j=1$, i.e., each part consists of two elements, and we need to pick one of them. Furthermore, we consider $k=\Gamma=1$. In the following, we show that this case can be solved in polynomial time. 

As $\Gamma = 1$, it is tempting to make a case distinction upon which item the adversary attacks. However, remember that the first-stage player first chooses one of exponentially many $\pmb x$, and then the adversary can react to this choice in a way not controllable by the first-stage player. Therefore, if we assume that a certain item is attacked, it is invalid to iterate over every single first-stage solution $\pmb x$ (because $\min_x \max_a f(x,a) \neq \max_a \min_x f(x,a)$ in general). Instead, as mentioned in \cref{sec:intro}, we do the following: For each possible attack $a$ of the adversarial player, we characterize the set $\X_a$ of first-stage solutions $\pmb x$ with the property that $a$ is an optimal response to $\pmb x$ (details below). We then show how the first-stage player can find the optimum of $\X_a$ in polynomial time. Surprisingly, the argument becomes more technical than one might expect at first.

For this section, we make use of the following notation. Let $\langle j,\pmb{x}\rangle\in[n]$ denote the index of the item chosen in part $j\in [K]$ by a fixed first-stage solution $\pmb{x}$, and let $\langle j,\overline{\pmb{x}}\rangle$ denote the index of the item that is not chosen. Using this notation, the incremental problem can be written in the following way.
\begin{subequations}
\label{eq:incs}  
\begin{align}
\textsc{Inc}(\pmb{x},\pmb{c}) = \min\ & \sum_{j\in[K]} c_{\langle j,\pmb{x}\rangle} + \sum_{j\in[K]} (c_{\langle j,\overline{\pmb{x}}\rangle}  - c_{\langle j, \pmb{x} \rangle}) y_j \label{incs-1} \\
\text{s.t. } & \sum_{j\in[K]} y_j \le 1 \label{incs-2}\\
& y_j \in \mathbb{N}_0 & \forall j\in[K] \label{incs-3}
\end{align}
\end{subequations}
We use a variable $y_j$ for every part $T_j$  ($j\in[K]$) to denote whether the item exchange takes place in this part. The objective~\eqref{incs-1} consists of two sums. The first sum denotes the costs if no element is changed. The second sum represents the effect of exchanging one item. There is only a single constraint~\eqref{incs-2}, enforcing that only a single item change can take place.

Relaxing and dualizing problem~(\ref{eq:incs}) gives the following adversarial problem:
\begin{subequations}
\label{eq:advs}  
\begin{align}
\textsc{Adv}(\pmb{x}) = \max\ & \sum_{j\in[K]} c_{\langle j, \pmb{x} \rangle} - \pi \label{advs-1} \\
\text{s.t. } & c_{\langle j, \pmb{x} \rangle} - c_{\langle j, \overline{\pmb{x}}\rangle} \le \pi & \forall j\in[K] \label{advs-2}\\
& \pmb{c}\in\cU \label{advs-3}\\
& \pi \ge 0 \label{advs-4}
\end{align}
\end{subequations}
Note that in an optimal solution to this problem, we can assume that $\pi$ is equal to the largest left-hand-side of constraints~(\ref{advs-2}) or equal to zero, if they are all negative. Hence, we find that
\begin{equation}
\textsc{Adv}(\pmb{x}) =  \max_{\pmb{c}\in\cU} \sum_{j\in[K]} c_{\langle j, \pmb{x} \rangle} - \max_{j\in[K]} [c_{\langle j, \pmb{x} \rangle} - c_{\langle j, \overline{\pmb{x}} \rangle}]_+ 
\label{eq:adv_problem_pj_1_short}
\end{equation}

A small technical difficulty is that there may exist multiple strategies for the adversary to solve $\textsc{Adv}(\pmb{x})$, which are all optimal. The following lemma guarantees that among two special strategies at least one is always optimal. Described informally, strategy I is the strategy to increase the price of an item currently not selected by the first-stage player, in order to decrease the value of a potential recovery at this item. Likewise, strategy II is the strategy to increase the price of an item currently selected by the first-stage player, forcing the first-stage player to either pay the increased price, or to recover this item (instead of recovering another item, which the first-stage player had preferred to recover if there had been no attack).

\begin{lem}
\label{lemma:strategy}
Given an instance of RRRSP with $n_j = 2$ and $p_j = 1$ for all $j$, let $\pmb x$ be a fixed first-stage solution. If $i^\star, j^\star$ have the property that
\begin{align}
j^\star &\in \arg\max \{ \underline{c}_{\langle j, \pmb{x}\rangle} - \underline{c}_{\langle j, \overline{\pmb{x}} \rangle} : j\in[K] \} \label{eq:jstar}\\
\text{and }\quad i^\star &\in \arg\max \{ d_{\langle j,\pmb{x}\rangle} - [\overline{c}_{\langle j, \pmb{x} \rangle}-\underline{c}_{\langle j,\overline{\pmb{x}}\rangle} - [\underline{c}_{\langle j^\star,\pmb{x}\rangle} - \underline{c}_{\langle j^\star, \overline{\pmb{x}}\rangle}]_+]_+  : j\in[K] \}
\label{eq:istar}
\end{align}
then there is an optimal solution $\pmb{c}^\star$ to $\textsc{Adv}(\pmb{x})$ as defined in equation~\eqref{eq:adv_problem_pj_1_short} where one of the following two cases holds:
\begin{enumerate}
\item (Strategy I) $c^\star_i = \begin{cases}
\overline{c}_i & \text{ if } i = \langle j^\star,\overline{\pmb{x}}\rangle \\
\underline{c}_i & \text{ otherwise } 
\end{cases} $ for all $i\in [n]$, or

\item (Strategy II) $c^\star_i = \begin{cases}
\overline{c}_i & \text{ if } i = \langle i^\star, \pmb{x} \rangle \\
\underline{c}_i & \text{ otherwise } 
\end{cases} $ for all $i\in[n]$.
\end{enumerate}
Furthermore, for any
\begin{equation} 
b^\star \in \arg\max \{ \underline{c}_{\langle j,\pmb{x}\rangle} - \underline{c}_{\langle j, \overline{\pmb{x}}\rangle} : j\in[K], j \neq j^\star\},
\label{eq:bstar}
\end{equation}
we have that
\begin{equation}
g_1 :=   [\underline{c}_{\langle j^\star,\pmb{x}\rangle} - \underline{c}_{\langle j^\star,\overline{\pmb{x}}\rangle}]_+ - \max\{[\underline{c}_{\langle j^\star,\pmb{x}\rangle} - \overline{c}_{\langle j^\star,\overline{\pmb{x}}\rangle}]_+,\ [\underline{c}_{\langle b^\star,\pmb{x}\rangle} - \underline{c}_{\langle b^\star,\overline{\pmb{x}}\rangle}]_+ \} 
\label{eq:g1}
\end{equation}
is the gain for the adversarial player, if the adversarial player employs strategy I instead of doing nothing at all. Similarily,
\begin{equation}
 g_2 := d_{\langle i^\star,\pmb{x}\rangle} - [\overline{c}_{\langle i^\star,\pmb{x}\rangle}-\underline{c}_{\langle i^\star,\overline{\pmb{x}}\rangle} - [\underline{c}_{\langle j^\star,\pmb{x}\rangle} - \underline{c}_{\langle j^\star,\overline{\pmb{x}}\rangle}]_+]_+
\label{eq:g2} 
\end{equation}
is the gain for the adversarial player, if the adversarial player employs strategy II instead of doing nothing at all.
\end{lem}
\begin{proof}
Recall that $\Gamma=1$, hence, the adversery can select only a single part $j$, where he increases the cost of one of the two items it contains. Note that there are only two possible strategies to optimize the adversarial value in \cref{eq:adv_problem_pj_1_short}:
\begin{enumerate}
\item (Strategy I) Reduce the value of the inner maximum. In this case, choose
$j^\star$ such that $j^\star$ satisfies \cref{eq:jstar}
and increase the costs of item $\langle j^\star,\overline{\pmb{x}}\rangle$. If the argmax in \cref{eq:jstar} has multiple elements, every choice is equally good. To see that \cref{eq:g1} is correct, take the difference of \cref{eq:adv_problem_pj_1_short} for the corresponding old and new values of $\pmb{c}$.

\item (Strategy II) Increase the value of the sum. Here one needs to consider that this may lead to an increase in the inner maximum, which should be avoided (i.e., we increase the cost of an item that will be dropped in the recovery step anyway). The best choice here is taking $i^\star$ such that $i^\star$ satisfies \cref{eq:istar} and increasing the costs of item $\langle i^\star,\pmb{x}\rangle$, which can again be seen as taking the difference of \cref{eq:adv_problem_pj_1_short} for old and new values of $\pmb{c}$. Again, if the argmax in \cref{eq:istar} has multiple elements, every choice is equally good. Substituting $i^\star$ into \cref{eq:istar} yields \cref{eq:g2}.

\end{enumerate}
\end{proof}

Note that the adversarial player will employ strategy I only if $g_1 \geq g_2$ and will employ strategy II only if $g_2\geq g_1$. Also note, that the numerical values of $g_1$, $g_2$ are independent of the concrete choices of $j^\star, i^\star, b^\star$ satisfying \cref{eq:jstar,eq:istar,eq:bstar}. We can therefore view $g_1, g_2$ as only dependent on $\pmb x$, which we make use of in the next proof. We can now prove the main result of this section.

\begin{theorem}\label{thm:easy_case_RRRS}
The recoverable robust representatives selection problem with discrete budgeted uncertainty and $\Gamma = k = 1$ and $p_j=1$, $n_j = 2$ for all $j\in[K]$ can be solved in strongly polynomial time $O(n^3)$.
\end{theorem}
\begin{proof}
Consider a fixed instance of the RRRSP with $\Gamma = k = 1$ and $n_j = 2$ for all $j$. For any $j^\star, i^\star, b^\star \in [K]$, with $j^\star \neq b^\star$, consider the following sets of first-stage solutions $\pmb x$: 
\begin{align*} \X_{j^\star b^\star} &:= \{ \pmb x\in \X : \pmb x \text{ satisfies } g_1 \geq g_2 \text{ and }\pmb x, j^\star, b^\star \text{ satisfy \cref{eq:jstar,eq:bstar}}\}\\
\mathcal{Y}_{j^\star i^\star} &:= \{ \pmb x\in \X : \pmb x \text{ satisfies } g_1 \leq g_2 \text{ and } \pmb x, j^\star, i^\star \text{ satisfy \cref{eq:jstar,eq:istar}}\}.
\end{align*}
It is easy to see that every first-stage solution is contained in at least one of the above sets. Hence, in order to prove the theorem, it suffices to show how to compute each of the values
$\min\{\textsc{Adv}(\pmb x) : \pmb x  \in \X_{j^\star b^\star}\}$ and $\min\{\textsc{Adv}(\pmb x) : \pmb x  \in \mathcal{Y}_{j^\star i^\star}\}$ in time $O(n)$. Taking the minimum of all obtained values yields the result.

\textbf{Computing the optimum of $\X_{j^\star b^\star}$}: 
For any $\pmb x \in \X_{j^\star b^\star}$ 
there are two possible items to choose from part $j^\star$ and part $b^\star$.
For each fixed choice of the four possible values of $\langle j^\star,\pmb{x}\rangle$ and $\langle b^\star,\pmb{x}\rangle$, run the following subroutine:
Observe that \cref{eq:jstar} and \cref{eq:bstar} are equivalent to the statement 
\begin{equation} \underline{c}_{\langle b^\star,\pmb{x}\rangle} - \underline{c}_{\langle b^\star,\overline{\pmb{x}}\rangle} \leq \underline{c}_{\langle j^\star,\pmb{x}\rangle} - \underline{c}_{\langle j^\star,\overline{\pmb{x}}\rangle} \ \land \ \forall j \in [K] \setminus \{j^\star, b^\star\} : \underline{c}_{\langle j,\pmb{x}\rangle} - \underline{c}_{\langle j, \overline{\pmb{x}}\rangle} \leq \underline{c}_{\langle b^\star,\pmb{x}\rangle} - \underline{c}_{\langle b^\star,\overline{\pmb{x}}\rangle} 
\label{eq:statement_1}
\end{equation}
and, under the condition that \cref{eq:statement_1} is true, we can compute $g_1$ by \cref{eq:g1}, and furthermore the statement $g_1 \geq g_2$ is equivalent to the statement 
\begin{equation}
\forall j \in [K] : d_{\langle j,\pmb{x}\rangle} - [\overline{c}_{\langle j,\pmb{x}\rangle}-\underline{c}_{\langle j, \overline{\pmb{x}}\rangle} - [\underline{c}_{\langle j^\star,\pmb{x}\rangle} - \underline{c}_{\langle j^\star,\overline{\pmb{x}}\rangle}]_+]_+ \leq g_1.
\label{eq:statement_2} 
\end{equation}
Therefore, we can iterate over $j \in [K]$ and, for each of the two possible choices of $\langle j,\pmb{x}\rangle$, check whether this choice satisfies \cref{eq:statement_1,eq:statement_2}. If this limits the possible choices to one value, we set $\langle j,\pmb{x}\rangle$ to this value. If this limits the possible choices to zero values, we can break the current loop iteration and skip to the next choice of $\langle j^\star,\pmb{x}\rangle$ and $\langle b^\star,\pmb{x}\rangle$. If both choices remain, we choose $\langle j,\pmb{x}\rangle$ such that $C_{\langle j,\pmb{x}\rangle} + \underline{c}_{\langle j,\pmb{x}\rangle}$ is minimum.

The result is a first-stage solution $\pmb{x}$ 
with total costs $g_1 + \sum_{j \in [K]} (C_{\langle j,\pmb{x}\rangle} + \underline{c}_{\langle j,\pmb{x}\rangle})$ by \cref{lemma:strategy}. 
Because $g_1$ only depends on $\langle j^\star,\pmb{x}\rangle$ and $\langle b^\star,\pmb{x}\rangle$, this value is minimum among all $\pmb x \in \X_{j^\star b^\star}$ considered in the current subroutine.

\textbf{Computing the optimum of $\mathcal{Y}_{j^\star i^\star}$:} 
Similarly to before, for each of the at most four choices of 
$\langle j^\star,\pmb{x}\rangle$ and $\langle i^\star,\pmb{x}\rangle$,
run the following subroutine:

Observe that \cref{eq:jstar} is equivalent to the statement
\begin{equation}
\forall j \in [K] \setminus \{ j^\star \} :  \underline{c}_{\langle j,\pmb{x}\rangle} - \underline{c}_{\langle j, \overline{\pmb{x}} \rangle} \leq \underline{c}_{\langle j^\star, \pmb{x}\rangle} - \underline{c}_{\langle j^\star,\overline{\pmb{x}}\rangle},
\label{eq:statement_3}
\end{equation}
and, under the condition that \cref{eq:statement_3} is true, \cref{eq:istar} is equivalent to the statement
\begin{align}
\forall j \in [K] \colon & \ d_{\langle j,\pmb{x}\rangle} - [\overline{c}_{\langle j,\pmb{x}\rangle}-\underline{c}_{\langle j, \overline{\pmb{x}}\rangle} - [\underline{c}_{\langle j^\star, \pmb{x}\rangle} - \underline{c}_{\langle j^\star, \overline{\pmb{x}}\rangle}]_+]_+ \nonumber \\ 
 & \leq d_{\langle i^\star,\pmb{x}\rangle} - [\overline{c}_{\langle i^\star,\pmb{x}\rangle}-\underline{c}_{\langle i^\star,\overline{\pmb{x}}\rangle} - [\underline{c}_{\langle j^\star,\pmb{x}\rangle} - \underline{c}_{\langle j^\star,\overline{\pmb{x}}\rangle}]_+]_+.
\label{eq:statement_4}
\end{align}
Likewise, under the condition that \cref{eq:statement_3,eq:statement_4} are true, $g_2$ can be computed by \cref{eq:g2}, and furthermore the statement $g_1 \leq g_2$ is equivalent to the statement
\begin{equation}
\forall j \in [K] \setminus{j^\star}: [\underline{c}_{\langle j^\star,\pmb{x}\rangle} - \underline{c}_{\langle j^\star,\overline{\pmb{x}}\rangle}]_+ - \max\{[\underline{c}_{\langle j^\star,\pmb{x}\rangle} - \overline{c}_{\langle j^\star,\overline{\pmb{x}}\rangle}]_+,\ [\underline{c}_{\langle j,\pmb{x}\rangle} - \underline{c}_{\langle j,\overline{\pmb{x}}\rangle}]_+ \} \leq g_2.
\label{eq:statement_5}
\end{equation}
Therefore, we can use \cref{eq:statement_3,eq:statement_4,eq:statement_5} to iterate over $j \in [K]$ analogously to the previous case.
At the end, we get a first-stage solution $\pmb{x}$ with total costs $g_2 + \sum_{j \in [K]} (C_{\langle j,\pmb{x}\rangle} + \underline{c}_{\langle j, \pmb{x}\rangle})$, by \cref{lemma:strategy}. 
Analogously to the previous case, we see that this is optimal among all first-stage solutions considered in the current subroutine.
\end{proof}

It seems likely that the proof of Theorem~\ref{thm:easy_case_RRRS} can be extended to general values $n_j$, if $p_j=\Gamma = k = 1$ holds. There still remain two basic adversarial strategies in this case; one where the chosen item of a part has its costs increased, and one where the cheapest item not chosen in a part has its cost increased. This may lead to an enumeration-based solution method along similar lines as presented in the proof. For the sake of brevity, we omit the details of this claim.

We further remark on two simple special cases. The first is for $k=0$. In this case, it is not possible to use any recovery action and \textsc{Rec} becomes
\[ \min_{\pmb{x}\in\X} \pmb{C}^t\pmb{x} + \max_{\pmb{c}\in\cU} \pmb{c}^t\pmb{x} 
=  \min_{\pmb{x}\in\X} \max_{\pmb{c}\in\cU} \left( (\pmb{C}+\pmb{c})^t \pmb{x} \right) \]
which is a standard min-max optimization problem with budgeted uncertainty. This means the results from~\cite{bertsimas2003robust} apply and \textsc{Rec} can be solved in polynomial time.

\begin{obs}
The recoverable robust representatives multi-selection problem with discrete budgeted uncertainty and $k=0$ can be solved in polynomial time.
\end{obs}

The second special case is for $\Gamma = 0$. Here, no adversarial action is possible, and problem \textsc{Rec} becomes
\[ \min_{\pmb{x}\in\X} \min_{\pmb{y}\in R(\pmb{x})} (\pmb{C}^t\pmb{x} + \underline{\pmb{c}}^t \pmb{y}) \]
This is a special case of the recoverable robust matroid basis problem with interval uncertainty studied in~\cite{lendl2019matroid},
where a strongly polynomial time algorithm for this problem is given. Our problem corresponds to the special case of partition matroids, 
hence, we arrive at the following observation.

\begin{obs}
The recoverable robust representatives multi-selection problem with discrete budgeted uncertainty and $\Gamma=0$ can be solved in strongly polynomial time.
\end{obs}

\section{Mixed-integer programming formulations}
\label{sec:mips}

In the following, we introduce several problem formulations for the RRRMSP. Two models use an exponential number of variables and constraints and can be used in combination with an iterative column-and-row generation procedure, while the third model is compact (i.e., has a polynomial number of variables and constraints).

\subsection{First model}
\label{sec:first}

A straight forward approach to model the RRRMSP is to observe that the set $\cU$ is discrete with $S=\binom{n}{\Gamma}$ many scenarios (in a worst-case solution, we can assume that $\sum_{i\in[n]} \delta_i = \Gamma$). With slight abuse of notation, let $\{ \pmb{\delta}^1,\ldots,\pmb{\delta}^S\} = \cU$ be a list of these scenarios. Problem \textsc{Rec} can then be modeled as the following mixed-integer program.
\begin{subequations}
\label{eq:m1}  
\begin{align}
\textsc{Rec} = \min\ & \sum_{i\in [n]} C_i x_i + t \label{m1-1} \\
\text{s.t. } & t \ge \sum_{i\in[n]} (\underline{c}_i + d_i\delta^s_i) y^s_i & \forall s\in[S], i\in[n] \label{m1-2} \\
& \sum_{i\in T_j} x_i = p_j & \forall j\in[K] \label{m1-3} \\
& \sum_{i\in T_j} y^s_i = p_j & \forall s\in[S], j\in[K] \label{m1-4} \\
& z^s_i \le x_i & \forall s\in[S], i\in[n] \label{m1-5} \\
& z^s_i \le y^s_i & \forall s\in[S], i\in[n] \label{m1-6} \\
& \sum_{i\in[n]} z^s_i \ge P-k & \forall s\in[S] \label{m1-7} \\
& x_i \in \{0,1\} & \forall i\in[n] \\
& y^s_i \in \{0,1\} & \forall s\in[S],i\in[n] \\
& z^s_i \in \{0,1\} & \forall s\in[S],i\in[n] \label{m1-end}
\end{align}
\end{subequations}
Variables $\pmb{y}^s$ are introduced for each scenario $s\in[S]$ to model the recovery action under this scenario. Constraints \eqref{m1-3} and \eqref{m1-4} ensure that the right number of items is chosen from each part $T_j$. Variable $z^s_i$ indicates whether item $i$ is not modified under scenario $s$, i.e., if both $x_i$ and $y^s_i$ are active (see constraints~\eqref{m1-5} and \eqref{m1-6}). The number of non-modified items must be at least $P-k$ by constraint~\eqref{m1-7}. In the objective~\eqref{m1-1} we minimize the sum of first-stage costs and the worst case over all second-stage costs. This is modeled using variable $t$, which must be greater or equal to the cost in each scenario using constraint~\eqref{m1-2}.

To avoid solving a model with an exponential number of variables and constraints, an iterative procedure is possible, where we begin with a subset of scenarios $\cU' = \{ \pmb{\delta}^1,\ldots,\pmb{\delta}^{S'}\} \subseteq \cU$ and alternate between solving model~(\ref{eq:m1}) and solving the adversarial problem with the current candidate solution $\pmb{x}$ to generate a new scenario (see, e.g., \cite{zeng2013solving}). The procedure ends after possibly exponentially many iterations if the lower bound (from solving (\ref{eq:m1}) with a subset of scenarios) and the upper bound (from solving the adversarial problem) coincide. 

This method can benefit from an efficient technique to solve the adversarial problem as studied in Section~\ref{sec:adv}.

\subsection{Second model}
\label{sec:second}

In Section~\ref{sec:adv} we demonstrated that there are discrete sets $B$ and $A_j(\beta)$ containing candidate values for variables $\beta$ and $\pmb{\alpha}$ of the adversarial problem. Let
\[ \mathcal{C} = \{ (\beta,\alpha_1,\ldots,\alpha_K) : \beta \in B, \alpha_j \in A_j(\beta) \} =  \{ (\beta^1,\pmb{\alpha}^1), \ldots, (\beta^S,\pmb{\alpha}^S) \} \]
contain all combinations of such candidate values. Note that set $\mathcal{C}$ is of exponential size. For a fixed choice $(\beta^s,\pmb{\alpha}^s)\in \mathcal{C}$, problem~(\ref{eq:ad}) becomes the following selection problem:
\begin{subequations}
\label{eq:ads}  
\begin{align}
\max\ & \sum_{j\in[K]} p_j \alpha^s_j + (P-k)\beta^s \nonumber \\
& - \sum_{j\in[K]} \sum_{i\in T_j} \left( \left[ \alpha_j^s + x_i\beta^s - \underline{c}_i \right]_+ (1-\delta_i) + \left[ \alpha_j^s + x_i\beta^s - \underline{c}_i - d_i \right]_+ \delta_i \right)  \label{ads-2}\\
\text{s.t. } & \sum_{i\in[n]} \delta_i \le \Gamma \label{ads-3}\\
& \delta_i \in\{0,1\} & \forall i\in[n] \label{ads-4}
\end{align}
\end{subequations}
Note that we can relax variables $\delta_i$ without changing the optimal objective value, as this is simply a selection problem. Hence, we can relax and dualize this problem to find the following formulation:
\begin{subequations}
\label{eq:secmod}  
\begin{align}
\min\ & \sum_{i\in[n]} C_i x_i + t \label{secmod-1} \\
\text{s.t. } & t \ge \sum_{j\in[K]} p_j\alpha^s_j + (P-k)\beta^s - \sum_{j\in[K]} \sum_{i\in T_j} [\alpha^s_j + x_i\beta^s - \underline{c}_i]_+ + \Gamma\pi^s + \sum_{i\in[n]} \rho^s_i \hspace*{-10mm}& \forall s\in[S] \\
& \pi^s + \rho^s_i \ge [\alpha^s_j + x_i\beta^s - \underline{c}_i]_+ - [\alpha^s_j + x_i \beta^s - \underline{c}_i - d_i]_+ & \forall s\in[S], i\in[n] \\
& \sum_{i\in T_j} x_i = p_j & \forall j\in[K] \\
& x_i \in\{0,1\} & \forall i\in[n] \\
& \pi^s \ge 0 & \forall s\in[S] \\
& \rho^s_i \ge 0 & \forall s\in[S],i\in[n] \label{secmod-end}
\end{align}
\end{subequations}
Variables $\pi^s$ and $\rho^s_i$ are dual variables for constraints~\eqref{ads-3} and \eqref{ads-4}, respectively. The new variable $t$ is introduced to model the worst-case value over all possible scenarios $s\in[S]$. As before, this problem can be linearized in $\pmb{x}$ by using $[a+bx_i]_+ = [a+b]_+ x_i + [a]_+(1-x_i)$. Similar to the formulation given in Section~\ref{sec:first}, this model contains an exponential number of variables and constraints. However, we do not model each choice of item cost increase explicitly, as before. Instead, the adversary is expressed by $(\beta^s,\pmb{\alpha}^s)$, and therefore takes the incremental problem into account as well. We can thus expect this formulation to be more effective in combination with an iterative variable-and-constraint generation procedure. We evaluate this experimentally in Section~\ref{sec:experiments}.

\subsection{Third model}
\label{sec:compact}

We now introduce a third approach to model the RRRMSP that makes use of the dynamic program introduced in the proof of Theorem~\ref{th:adv} to require only polynomially many variables and constraints.
As explained in Section~\ref{sec:second}, we can dualize the adversarial problem~\eqref{eq:ads} for a fixed choice of $\beta$ and $\pmb{\alpha}$ to find:
\begin{align*}
\textsc{Adv}_j(\pmb{x},\beta,\Gamma_j,\alpha_j) = \min\ &(P-k)\beta + p_j\alpha_j - \sum_{i\in T_j} [\alpha_j + x_i\beta - \underline{c}_i]_+ + \Gamma_j \pi + \sum_{i\in T_j} \rho_i \\
\text{s.t. } & \pi + \rho_i \ge [\alpha_j + x_i\beta - \underline{c}_i]_+ - [\alpha_j + x_i \beta - \underline{c}_i - d_i]_+ & \forall i\in T_j \\
& \pi \ge 0 \\
& \rho_i \ge 0 & \forall i\in T_j
\end{align*}
As $\textsc{Adv}_j(\pmb{x},\beta,\Gamma_j) = \max_{\alpha\in A_j(\beta)} \textsc{Adv}_j(\pmb{x},\beta,\Gamma_j,\alpha)$, we find that
\begin{align*}
\textsc{Adv}_j(\pmb{x},\beta,\Gamma_j) = &(P-k)\beta \\
+ \min\ & t_j \\
\text{s.t. } & t_j \ge p_j \alpha - \sum_{i\in T_j} [\alpha + x_i\beta - \underline{c}_i]_+ + \Gamma_j \pi_j^\alpha + \sum_{i\in T_j} \rho_i^\alpha & \forall \alpha \in A_j(\beta) \\
& \pi_j^\alpha + \rho_i^\alpha \ge [\alpha + x_i\beta - \underline{c}_i]_+ - [\alpha + x_i \beta - \underline{c}_i - d_i]_+ & \forall i\in T_j, \alpha\in A_j(\beta)\\
& \pi_j^\alpha \ge 0 & \forall \alpha\in A_j(\beta) \\
& \rho_i^\alpha \ge 0 & \forall i\in T_j , \alpha\in A_j(\beta)
\end{align*}
This model enables us to evaluate $\textsc{Adv}_j(\pmb{x},\beta,\Gamma_j)$ during the optimization process (instead of during a preprocessing step as performed for the dynamic program). Each state of the dynamic program can now be interpreted as a node $(K',\Gamma')$ in a directed acyclic graph. The aim of the adversarial problem is to find a path from an artificial node $(0,0)$ (representing the starting state) to the node $(K,\Gamma)$ (representing the state when at most $\Gamma$ many costs have been increased up to and including part $T_K$). Traversing an edge from $(K',\Gamma')$ to $(K'+1,\Gamma'')$ with $\Gamma''\ge \Gamma'$ means that a budget of $\Gamma''-\Gamma'$ is invested in part $K'+1$. The corresponding longest path problem can be relaxed and dualized, which gives the following compact formulation for $\textsc{Rec}$:
\begin{subequations}
\label{eq:com}  
\begin{align}
\min\ & \sum_{i\in[n]} C_i x_i + t \label{com-1} \\
\text{s.t. } & t \ge (P-k)\beta + s^\beta_{K+1,\Gamma} & \forall \beta \in B \label{com-2}\\
& s^\beta_{j+1,\gamma'} \ge s^\beta_{j,\gamma} + c^\beta_{j,\gamma'-\gamma} & \forall \beta \in B, \gamma\in 0\,\ldots,\Gamma, \gamma'\ge \gamma, j \in [K] \label{com-3}\\
& s^\beta_{K+1,\gamma} \ge s^\beta_{K+1,\gamma-1} & \forall \beta\in B, \gamma\in[\Gamma] \label{com-4}\\
& s^\beta_{1,\gamma} =0 & \forall \beta\in B, \gamma\in \{0,\ldots,\Gamma\} \label{com-5}\\
& c^\beta_{j,\gamma} \ge  p_j \alpha - \sum_{i\in T_j} [\alpha + x_i\beta - \underline{c}_i]_+ + \gamma \pi^{\beta,\gamma,\alpha}_j + \sum_{i\in T_j} \rho^{\beta,\gamma,\alpha}_i \hspace{-5cm} \nonumber\\
& & \forall \beta \in B, j \in [K], \gamma \in \{0,\ldots,\Gamma\}, \alpha \in A_j(\beta) \label{com-6} \\
& \pi^{\beta,\gamma,\alpha}_j + \rho^{\beta,\gamma,\alpha}_i \ge [\alpha + x_i\beta - \underline{c}_i]_+ - [\alpha + x_i \beta - \underline{c}_i - d_i]_+ \hspace{-5cm} \nonumber\\
& & \forall \beta\in B, j\in [K], \gamma\in \{0,\ldots,\Gamma\}, \alpha \in A_j(\beta), i\in T_j \label{com-7}\\
& \sum_{i\in T_j} x_{ji} = p_j & \forall j\in [K] \label{com-8}\\
& x_{ji} \in\{0,1\} & \forall j\in[K], i\in T_j \label{com-9}\\
& \pi^{\beta,\gamma,\alpha}_j \ge 0 & \forall \beta\in B, j\in[K], \gamma\in\{0,\ldots,\Gamma\}, \alpha\in A_j(\beta) \label{com-10} \\
& \rho^{\beta,\gamma,\alpha}_i \ge 0 & \forall \beta\in B, j\in[K], \gamma\in\{0,\ldots,\Gamma\}, \alpha\in A_j(\beta), i\in T_j \label{com-11}
\end{align}
\end{subequations}
We make use of variables $c^\beta_{j,\gamma}$ that denote the costs $\textsc{Adv}_j(\pmb{x},\beta,\gamma)$. Variables $s^\beta_{j,\gamma}$ are the dual variables corresponding to the nodes of the graphs (node potentials). For ease of notation, the index $j$ is increased by one.
The objective~\eqref{com-1} is to minimize the sum of first-stage costs and worst-case second-stage costs, represented by variable $t$. In constraints~\eqref{com-2} we ensure that $t$ is the maximum over all candidate values $\beta$ for the adversarial problems. The right-hand side of \eqref{com-2} uses variable $s^\beta_{K+1,\Gamma}$ to express the value of the longest path in the graph modeling the adversarial problem. Constraints~(\ref{com-3}-\ref{com-5}) model the dual constraints to this problem, where $c^\beta_{j,\gamma}$ represents the arc lengths. These are variables, which depend on the worst-case choice of $\alpha\in A_j(\beta)$. This is ensured by constraints~\eqref{com-6} and \eqref{com-7}. As before, brackets $[\cdot]_+$ can be linearized using standard methods. Finally, constraint~\eqref{com-9} ensures that $\pmb{x}\in\X$. Note that the size of the model is polynomial in the input with $O(n^4\Gamma)$ many variables and constraints.

Next we consider the integrality gap of this formulation. In Table~\ref{tab:example2} we present an example instance consisting of two parts, from each of which one item must be selected. As before, $\Gamma=k=1$.
\begin{table}[htb]
\begin{center}
\begin{tabular}{r|rrr|rr}
 & \multicolumn{3}{c|}{$T_1$} & \multicolumn{2}{c}{$T_2$} \\
 & 1 & 2 & 3 & 4 & 5 \\
 \hline
$C_i$ & 0 & 0 & 0 & 0 & 1\\
$\underline{c}_i$ & 0 & 0 & 0 & 1 & 0 \\
$\overline{c}_i$ & 1 & 1 & 1 & 1 & 0
\end{tabular}
\end{center}
\caption{Example problem with $\Gamma=k=1$.\label{tab:example2}}
\end{table}
Note that an optimal integral solution has objective value 1. Let us assume without loss of generality that item 1 is packed. If also item 4 is packed, the adversary attacks item 1, and either exchanging item 1 or item 4 results in an objective value 1. If item 5 is packed instead of item 4, we must pay one unit in the first stage but can remain without additional costs in the second stage.

Now consider the fractional solution $\pmb{x}=(1/3,1/3,1/3,2/3,1/3)$. We have $B=\{-1,0,1\}$, and careful examination of all problems $\textsc{Adv}_j(\pmb{x},\beta,\Gamma_j)$ reveals that none of these results in an adversary value larger than 0. We can therefore make the following observation.

\begin{obs}
The compact problem formulation (\ref{eq:com}) has an unbounded integrality gap.
\end{obs}

This means in particular that we cannot expect algorithms based on rounding the linear relaxation of the compact formulation to result in approximation guarantees. A natural question is thus to consider whether local search methods may provide an approximation guarantee.

\section{Experiments}
\label{sec:experiments}

\subsection{Setup}

The aim of this section is to compare the behavior of the exact solution methods described in Section~\ref{sec:mips} depending on instance parameters. The models we compare are:
\begin{enumerate}
\item Model (\ref{eq:m1}) from Section~\ref{sec:first} in combination with iterative variable-and-constraint generation. We refer to this approach as M1.
\item Model (\ref{eq:secmod}) from Section~\ref{sec:second} in combination with iterative variable-and-constraint generation. We refer to this approach as M2.
\end{enumerate}
Note that the compact model (\ref{eq:com}) from Section~\ref{sec:compact} is a natural additional comparator approach. However, preliminiary experiments show that this approach does not scale well in the problem size due to the large number of variables. Hence, it is not considered in the following experiments.

To explain the experimental setup, we first consider the number of scenarios $S$ that can potentially be generated using approaches M1 and M2. As noted in Section~\ref{sec:first}, we can bound the number of scenarios $S$ considered in M1 by $\binom{n}{\Gamma}$, which is in $O(n^\Gamma)$. The model uses up to $O(nS)$ many variables and constraints. To estimate
\[ \mathcal{C} = \{ (\beta,\alpha_1,\ldots,\alpha_K) : \beta \in B, \alpha_j \in A(\beta) \} =  \{ (\beta^1,\pmb{\alpha}^1), \ldots, (\beta^S,\pmb{\alpha}^S) \} \]
and thus the maximum number of iterations for M2, we note that
\[ |\mathcal{C}| = |B| \prod_{j\in[K]} |A(\beta)| = O(n^2 \prod_{j\in[K]} |T_j|) = O(n^2\left(\frac{n}{K}\right)^K ) \]
M2 uses $O(n|\mathcal{C}|)$ up to many variables and constraints.
We can therefore expect a sensitivity of M1 with respect to parameter $\Gamma$, while the efficiency of M2 will be less dependent on $\Gamma$, and more sensitive to changes in $K$. We therefore generate two sets of instances:
\begin{itemize}
\item Set $I_1$, where $K=10$ and $n_j=10$ for all $j\in[K]$, and $k=\Gamma/2$. We consider $\Gamma\in \{2,4,6,\ldots,100\}$ and generate 200 instances for each value of $\Gamma$ (a total of $10,000$).

\item Set $I_2$, where $n_j=3$ for all $j\in[K]$, $\Gamma=K$, and $k=\Gamma/2$. We consider $K\in \{2,4,6,\ldots,60\}$ and generate 200 instances for each value of $K$ (a total of $6,000$).
\end{itemize}
For all instances, values $C_{ij}$, $\underline{c}_{ij}$ and $d_{ij}$ are chosen uniformly i.i.d.\! from $\{1,2,\ldots,100\}$. Values $p_j$ are chosen uniformly i.i.d.\! from $\{1,2,\ldots,n_j-1\}$.

Each instance is solved using M1 and M2, where we apply a time limit of 900 seconds. Models are solved using IBM CPLEX 12.8 on a virtual machine with Intel Xeon CPU E7-2850 processors using only one thread. Both types of master- and sub-problems were solved using CPLEX.

\subsection{Results}

We first discuss the results for instances $I_1$, where only $\Gamma$ is modified and the remaining instance size is left fixed. In Figure~\ref{fig:opt} we show the fraction of instances (out of 200) per parameter value that can be solved to optimality within the time limit. The case of instances $I_1$ is presented in Figure~\ref{fig:opt1}. A V-shaped decrease in instances solved to optimality is visible for M1, while this is not present for M2 (only a single instance was not solved to optimality using M2). The worst performance for M1 is at $\Gamma=16$, where only 113 instances could be solved. Note that a V-shape is not unexpected, as $\binom{n}{\Gamma}=\binom{n}{n-\Gamma}$. Figure~\ref{fig:opt1} demonstrates that M2 is less affected by $\Gamma$ than M1. 

\begin{figure}[htb]
\begin{center}
\subfigure[Instances $I_1$.\label{fig:opt1}]{\includegraphics[width=0.48\textwidth]{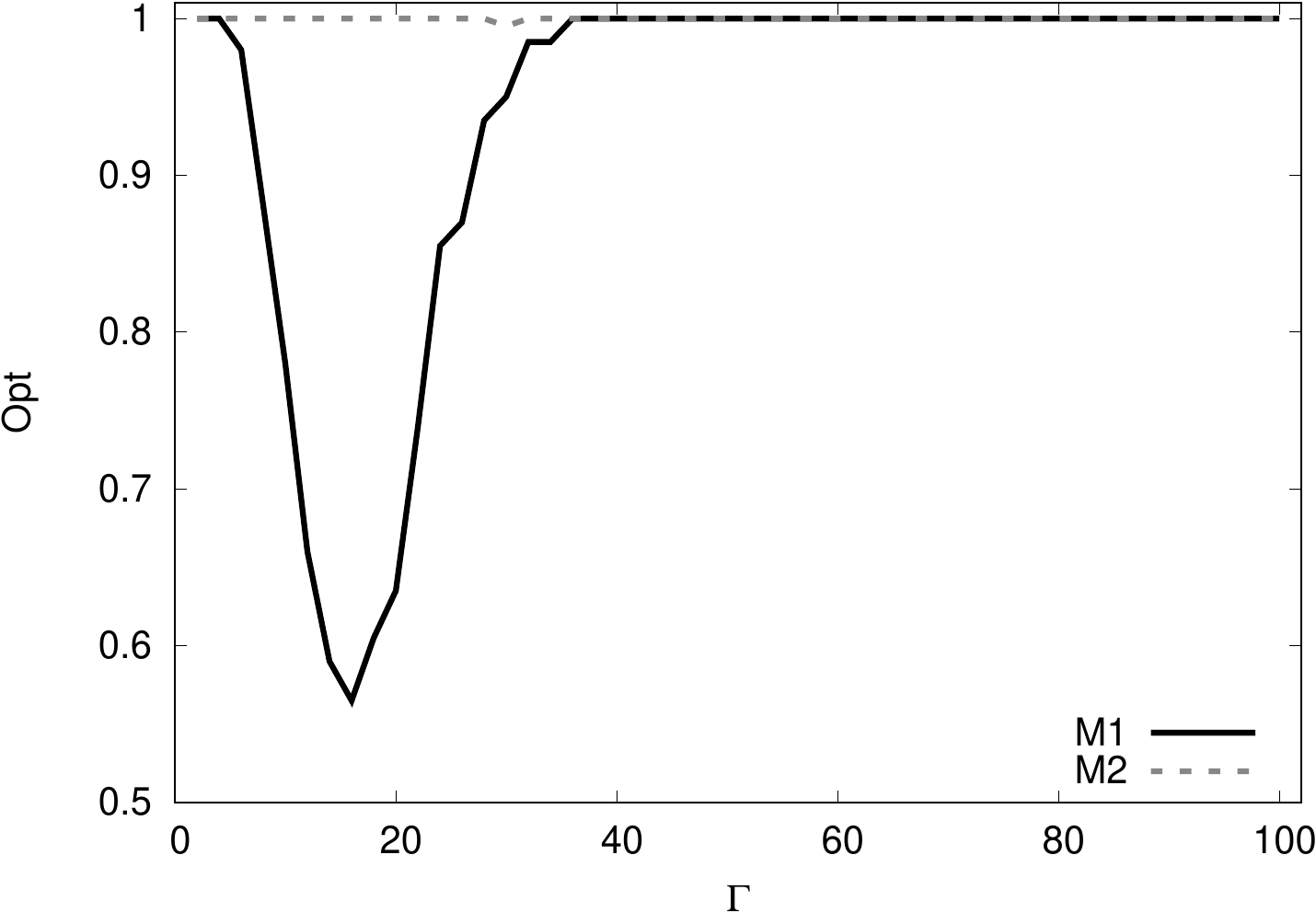}}%
\hfill
\subfigure[Instances $I_2$.\label{fig:opt2}]{\includegraphics[width=0.48\textwidth]{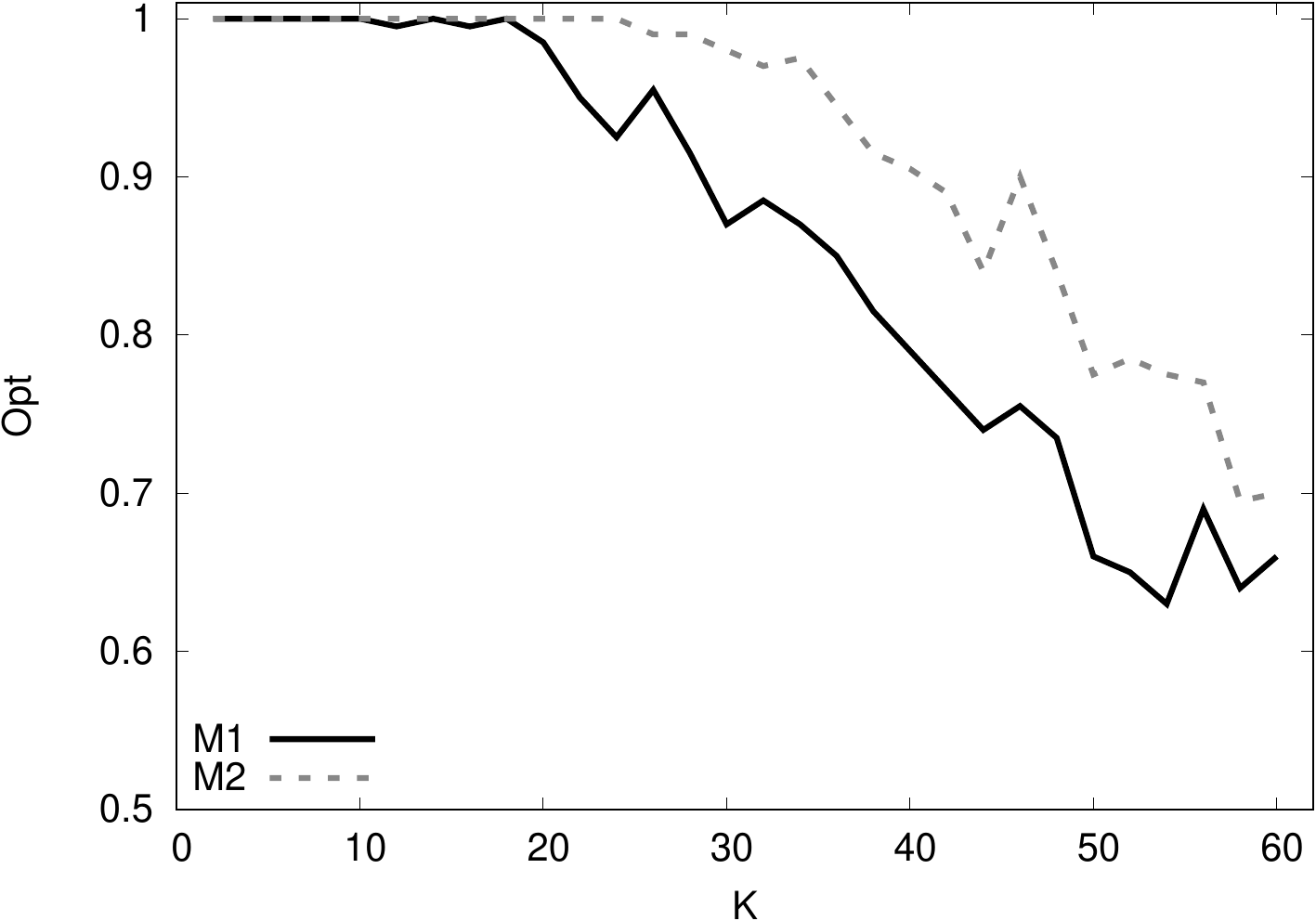}}
\end{center}
\caption{Ratio of instances solved to optimality for varying $\Gamma$ ($I_1$) and $K$ ($I_2$).\label{fig:opt}}
\end{figure}

This can be seen in more detail in Figure~\ref{fig:time1}, where we show the average computation time for M1 and M2. Any instance that could not be solved to optimality counts as 900 seconds towards this average. Note the logarithmic vertical axis. The peak in computation times for M2 around $\Gamma=16$ is visible again; at the same time, we see that also approach M2 follows a similar curve as M1, while much less severe.

\begin{figure}[htb]
\begin{center}
\subfigure[Instances $I_1$.\label{fig:time1}]{\includegraphics[width=0.48\textwidth]{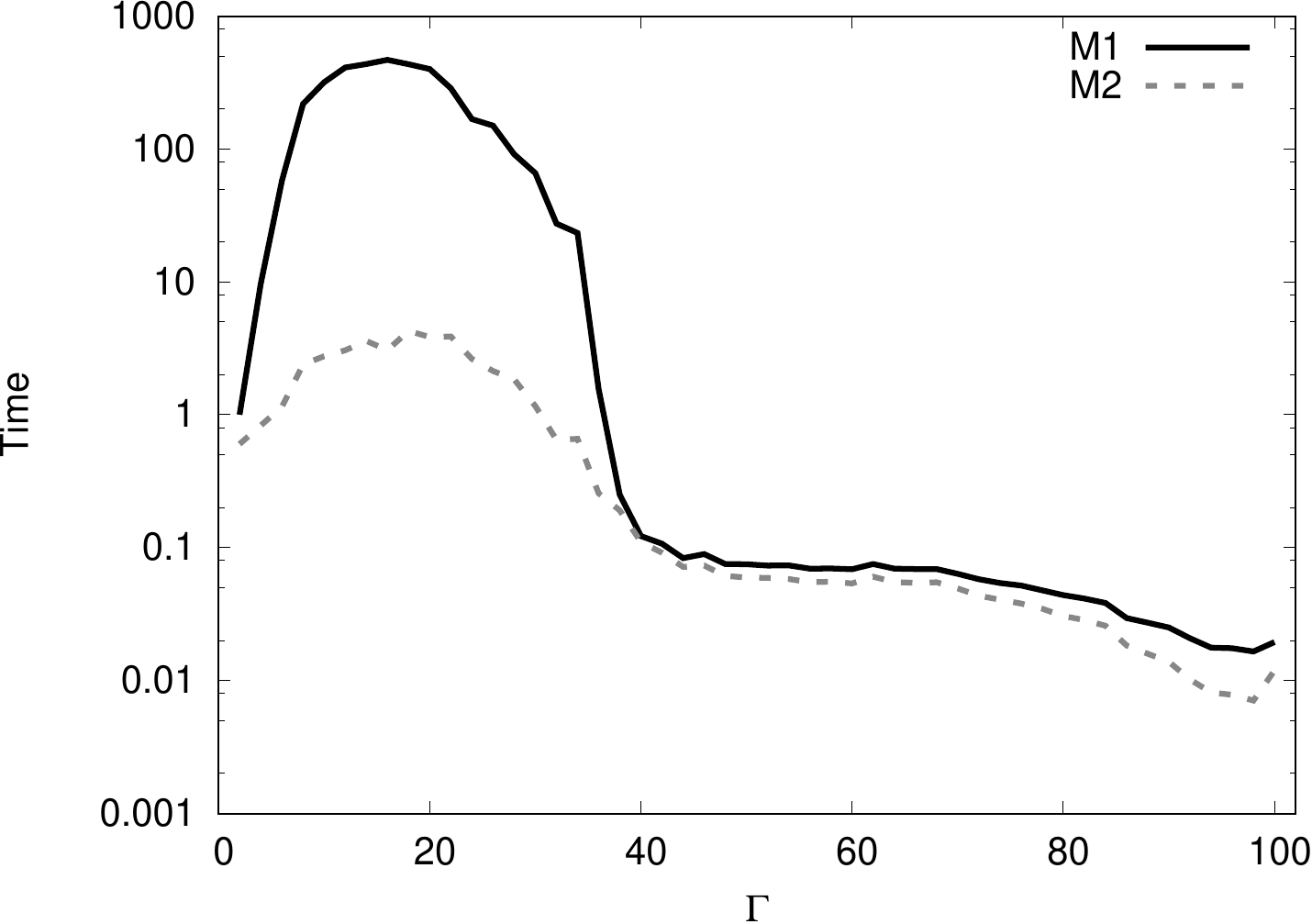}}%
\hfill
\subfigure[Instances $I_2$.\label{fig:time2}]{\includegraphics[width=0.48\textwidth]{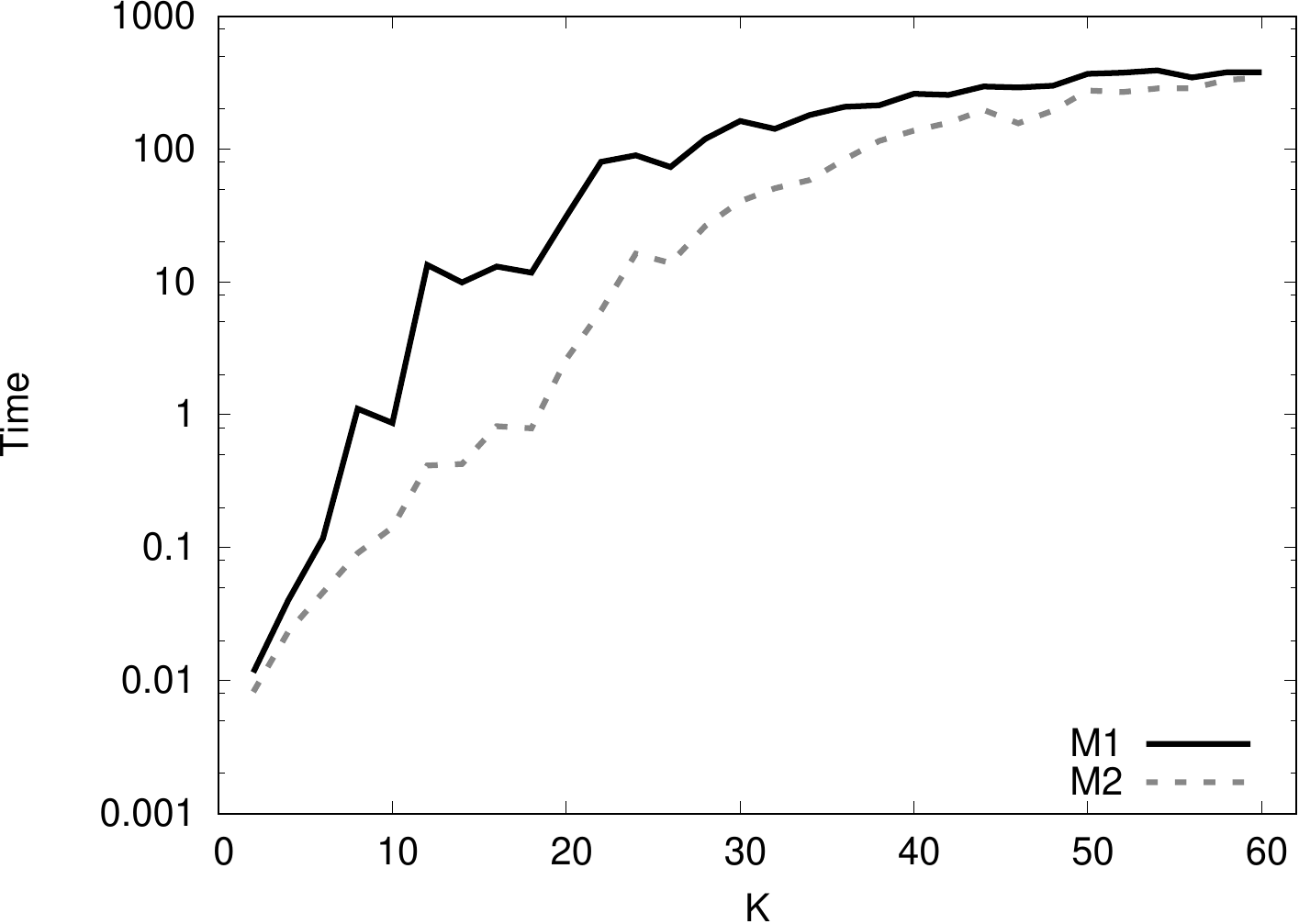}}
\end{center}
\caption{Average truncated solution time in seconds for varying $\Gamma$ ($I_1$) and $K$ ($I_2$).\label{fig:time}}
\end{figure}	

A direct comparison between computation times for each instance is given in Figure~\ref{fig:tvst1}. Here, each point corresponds to one of the $10,000$ instances (note that both axes are logarithmic). A point below the diagonal indicates that M2 requires less time than M1 on this instance. For only few instances of $I_1$ it is faster to use M1 instead of M2, and none of these require more than one second to solve.

\begin{figure}[htb]
\begin{center}
\subfigure[Instances $I_1$.\label{fig:tvst1}]{\includegraphics[width=0.48\textwidth]{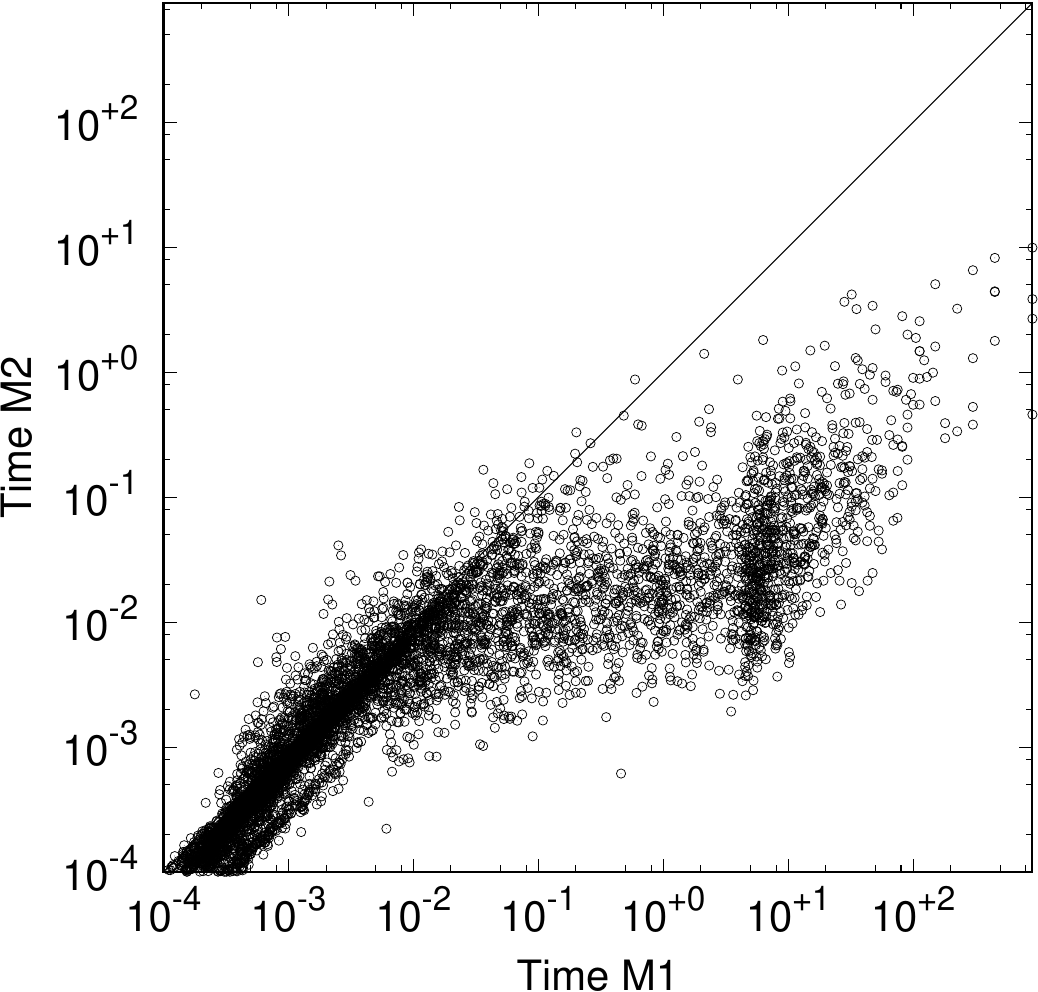}}%
\hfill
\subfigure[Instances $I_2$.\label{fig:tvst2}]{\includegraphics[width=0.48\textwidth]{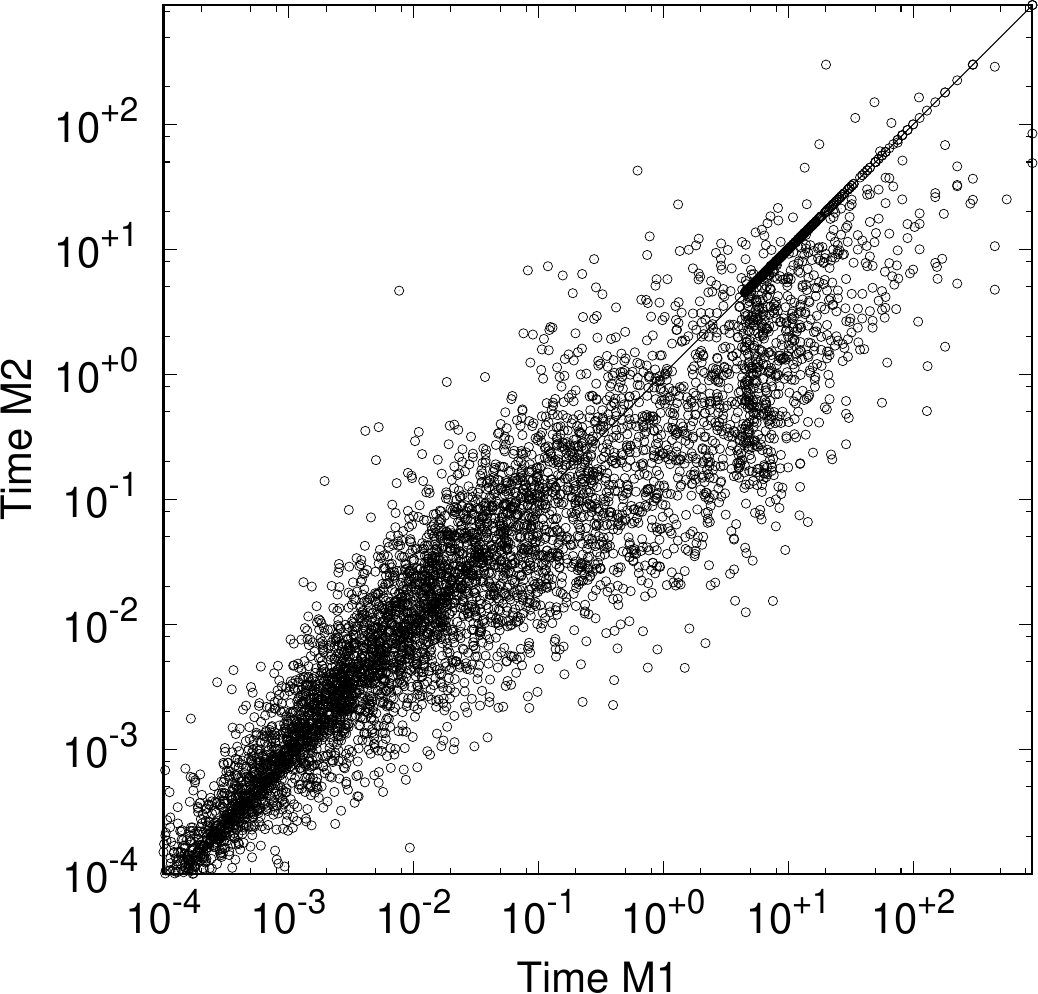}}
\end{center}
\caption{Instance-by-instance comparison of solution times.\label{fig:tvst}}
\end{figure}

Finally, we also present the average number of iterations required to prove optimality by each method in Figure~\ref{fig:it1}. In this average we only consider those instances that were solved to optimality by both methods. The minimum number of iterations is 2 (lower and upper bounds coincide after generating one scenario), and is achieved on all instances by both methods for $\Gamma \ge 54$. Instances with $\Gamma=16$ require an average of $18.25$ iterations to solve by M1, while this remains at $9.66$ for M2.

\begin{figure}[htb]
\begin{center}
\subfigure[Instances $I_1$.\label{fig:it1}]{\includegraphics[width=0.48\textwidth]{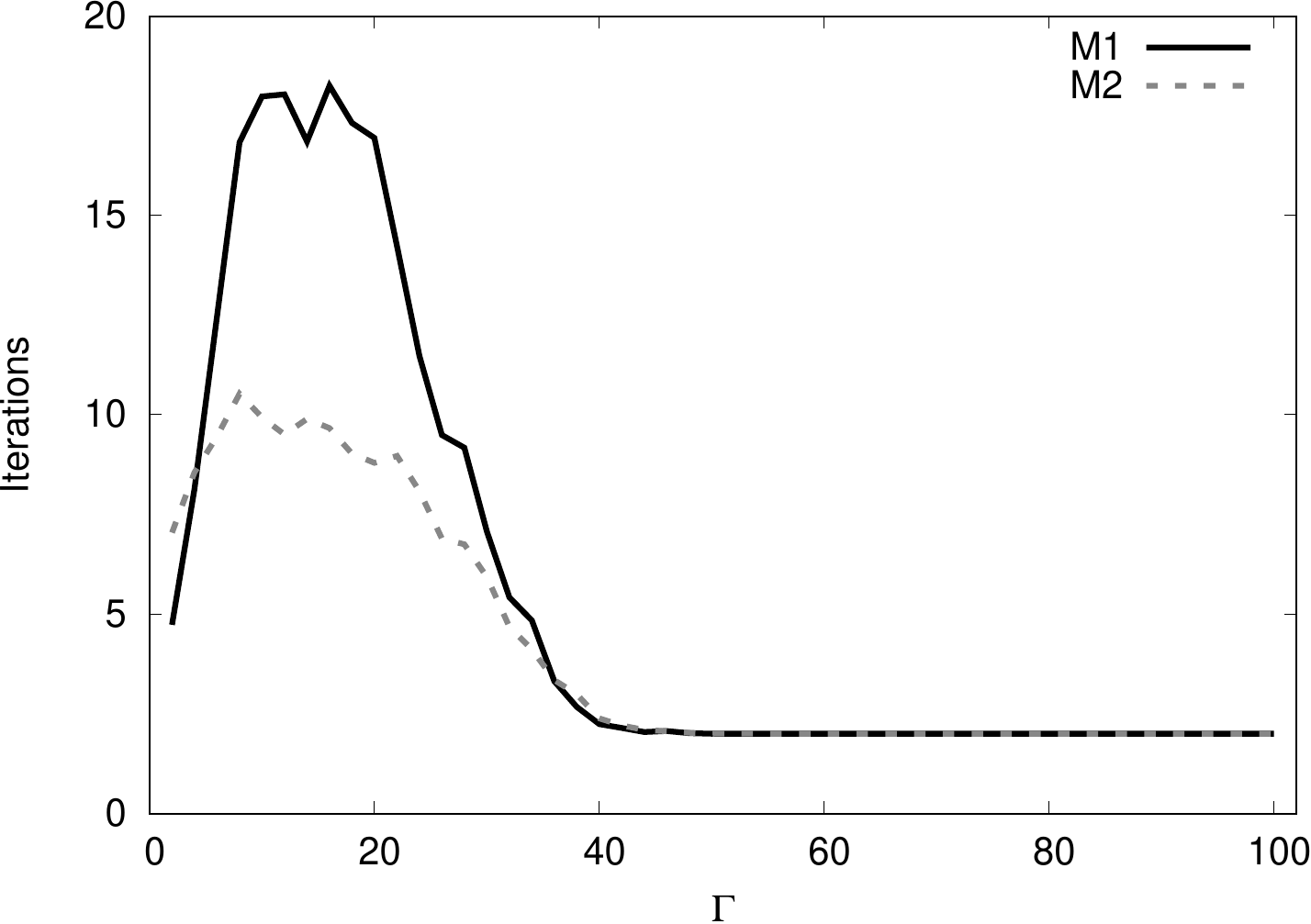}}%
\hfill
\subfigure[Instances $I_2$.\label{fig:it2}]{\includegraphics[width=0.48\textwidth]{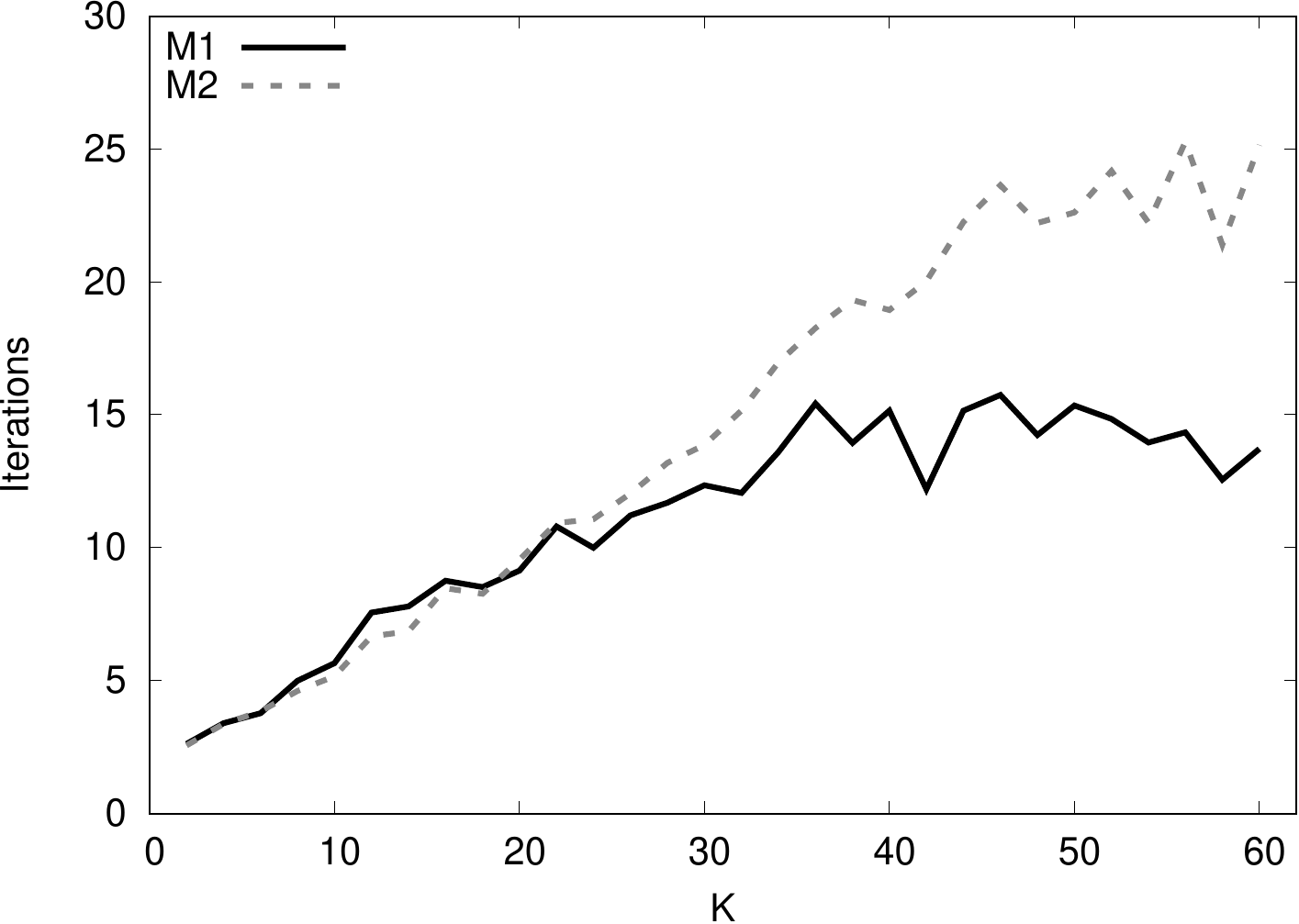}}
\end{center}
\caption{Average number of iterations for varying $\Gamma$ ($I_1$) and $K$ ($I_2$).\label{fig:it}}
\end{figure}	

We now compare these findings with the results for instance set $I_2$. Note that with increasing value $K$, also the number of items $n$, and the parameters $\Gamma$ and $k$ increase. We therefore expect both M1 and M2 to perform worse for increasing $K$. Considering the number of instances solved to optimality as presented in Figure~\ref{fig:opt2}, this can be confirmed. The number of instances solved to optimality by M1 remains lower than those solved by M2. The average computation times as shown in Figure~\ref{fig:time2} also increase in $K$, where M2 is faster on average. A direct comparison of computation times for each instance as presented in Figure~\ref{fig:tvst2} shows that points are much more closely aligned to the diagonal than in Figure~\ref{fig:tvst1}. This indicates that M2 remains the better approach for most instances, but results tend to be more mixed. In one specific instance, M2 required 300 seconds, while M1 only required 20 seconds.

Considering the average number of iterations given in Figure~\ref{fig:it2}, we find another difference to instance set $I_1$. Up to around $K=20$, both approaches require a similar number of iterations. For larger values of $K$, this number keeps on increasing for M2, whereas this is not the case for M1. In comparison with Figure~\ref{fig:it1}, we see that the maximum average number of iterations for M1 is even smaller for $I_2$ than for $I_1$.

To summarize these findings, we note that M2 can solve all instances (but one) of size $10\times 10$ in a matter of a few seconds, and clearly outperforms M1 in these cases, particularly for parameters $\Gamma$ in the region of $[\sqrt{n},3\sqrt{n}]$. Due to the way scenarios are generated, the performance of M2 depends less on $\Gamma$ than the performance of M1. For instances of size $K \times 3$ with increasing $K$, this advantage of M2 is less pronounced. While M2 continues to show the better performance, results tend to be more mixed.

Overall, M2 performs better, which highlights the importance of the structural insight into the adversarial problem gained in Section~\ref{sec:adv}.

\section{Conclusions}
\label{sec:conclusions}

In this paper we considered the following robust variant of the representatives selection problem. Given $K$ sets of items called parts, we first choose $p_j$ items from each part $j\in[K]$. Then an adversary creates a cost scenario by increasing the costs of up to $\Gamma$ many items. We can now update our solution by exchanging up to $k$ items over all parts. The aim is to minimize the overall costs given by the costs of the first-stage solution, and the worst-case costs of the second-stage solution.

While this problem has been introduced nearly a decade ago, its complexity has remained open. This is perhaps not surprising, as an optimal solution can be quite counter-intuitive. An example provided in this paper demonstrates that it may be necessary to choose a dominated item in the first stage, i.e., an item that is worse with respect to every cost vector compared to another item from the same part.

We showed that it is possible to solve the adversarial problem in polynomial time and used this result to derive a compact mixed-integer programming formulation for the recoverable robust problem. We further prove that the recoverable robust problem is NP-hard, even if $n_j=2$ and $p_j=1$ for all parts $j\in[K]$. We also showed that the special case with $K=1$ is NP-hard, even if $k=1$. In the special case that $k=\Gamma=1$, this problem allows a polynomial time solution algorithm. This method is based on identifying two basic strategies that represent all possible attacks of the adversary. We can then construct a solution for each of these strategies, such that this strategy is in fact preferred by the adversary over the other. The better of these two solutions is optimal for the recoverable robust problem.

In computational experiments we compared two iterative methods that scale differently in the problem parameters. While the straight-forward scenario generation method is sensitive to the size of $\Gamma$, a second method that is based on structural insight into the adversarial problem can be expected to be more sensitive to the number of parts $K$.

Many interesting avenues for further research arise. 
While these results cover many complexity aspects of the problem, the complexity for constant $\Gamma$ remains open. Additionally, our reductions are based on weakly NP-hard problems, which means that it remains an open question if pseudopolynomial solution methods exist.
Beyond the problem considered here, the complexity of other problem variants remains open as well. These include two-stage representatives selection (where an incomplete solution is built in the first stage) or other uncertainty sets, such as continuous budgeted uncertainty.

\appendix

\section{Proof of Theorem~\ref{th:hardness3}}

We first recall the defition of the two-stage robust selection problem from \cite{chassein2018recoverable}. Let $\X' = \{ \pmb{x}\in\{0,1\}^n : \sum_{i\in[n]} x_i \le p\}$ the set of feasible first-stage solutions. Given $\pmb{x}\in\X'$, the set of feasible second-stage solutions is then $\X(\pmb{x}) = \{ y\in\{0,1\}^n : \sum_{i\in[n]} x_i + y_i = p,\ x_i+y_1 \le 1 \ \forall i\in[n]\}$. As before, let $\cU$ be a discrete budgeted uncertainty set. The two-stage robust selection problem is then given as
\[ \min_{\pmb{x}\in\X'} \max_{\pmb{c}\in\cU} \min_{y\in\X(\pmb{x})} \pmb{C}^t\pmb{x} + \pmb{c}^t\pmb{y} \]
We now prove hardness of this problem (Theorem~\ref{th:hardness3}).

\begin{proof}
Let an instance of \textsc{Partition} be given, consisting of a list of values $A=\{a_1,\ldots,a_n\}$. The task is to identify a set $A_0\subseteq A$ such that $\sum A_0 = 1/2 \sum A =: Q$. We construct an instance of the two-stage robust selection problem using $2n$ items, $p=n+1$, and $\Gamma=n$. 

\begin{table}[htb]
\begin{center}
\begin{tabular}{c|ccc|ccc}
 \multicolumn{1}{c}{} & \multicolumn{3}{c}{$\alpha$}& \multicolumn{3}{c}{$\beta$}\\[-1ex]
 \multicolumn{1}{c}{} & \multicolumn{3}{c}{\downbracefill} & \multicolumn{3}{c}{\downbracefill}\\[2ex]
 & 1 & $\dots$ & $n$ & $n+1$ & $\dots$ & $2n$ \\
 \hline
$C_i$ & $\infty$ & $\dots$ & $\infty$ & $a_1$ & $\dots$ & $a_n$ \\
$\underline{c}_i$ & $M$ & $\dots$ & $M$ & 0 & $\dots$ & 0 \\
$\overline{c}_i$ & $M+2Q$ & $\dots$ & $M+2Q$ & $2a_1$ & $\dots$ & $2a_n$
\end{tabular}
\end{center}
\caption{Instance used in the hardness reduction for two-stage robust selection.\label{tab:h3}}
\end{table}

There are $n$ items of type $\alpha$ and of type $\beta$. Items of type $\alpha$ all have $C_i=\infty$, $\underline{c}_i=M$ for a sufficiently large constant $M > 2Q$, and $\overline{c}_i = M+2Q$. Here, $\infty$ denotes a large value such that packing such an item immediately disqualifies a solution from being optimal. As the following analysis shows, $\infty > M+4Q$ is sufficient.
Items of type $\beta$ each correspond to one of the given values $a_i$, with $C_i = a_i$, $\underline{c}_i=0$, and $\overline{c}_i=2a_i$. Note that an optimal first-stage solution does not buy any items of type $\alpha$. The adversary now has two strategies to distribute the uncertainty budget $\Gamma$: Either the costs of all items in $\alpha$ are increased, or all costs of items in $\beta$ are increased. In both cases, a second-stage solution will buy the remaining items in $\beta$ and one item in $\alpha$. In the first strategy, the costs of this item from $\alpha$ is $M+2Q$, in the second case it is $M$. Note that other strategies do not need to be considered, as any other distribution of the uncertainty also leads to costs $M$ for the one item of type $\alpha$ that is bought in the second stage.

Let $X$ be the first-stage costs of items of type $\beta$ that are packed by some solution. The total costs of this solution becomes
\[ X + \max\{ 4Q - 2X +M , 2Q+M \} = M + 2Q + \max\{ 2Q-X, X \} \]
which are less or equal to $M+2Q$ if and only if $X=Q$, i.e., if \textsc{Partition} is a yes-instance.

\end{proof}

\end{document}